\documentclass[11pt]{amsart}

\usepackage{mathrsfs, amsmath, graphics}
\usepackage{caption}
\usepackage{subfigure}
\usepackage{graphicx}
\usepackage{mathrsfs}
\usepackage{epsfig}
\usepackage{amsmath}
\usepackage{amsfonts}
\usepackage{amssymb}
\usepackage{amssymb,amscd}
\usepackage{tikz}
\usepackage{verbatim}
\usepackage[title]{appendix}
\usepackage{subfigure}
\usepackage{floatrow}
\newfloatcommand{capbtabbox}{table}[][\FBwidth]
\usepackage{lineno}

\newtheorem{thm}{Theorem}[section]

\newtheorem{prop}[thm]{Proposition}

\newtheorem{conj}[thm]{Conjecture}
\theoremstyle{definition}
\newtheorem{defn}[thm]{Definition}

\theoremstyle{remark}

\title{Three-manifolds at infinity of  complex hyperbolic orbifolds}

\author{Jiming Ma}
\address{School of Mathematical Sciences, Fudan University, Shanghai, 200433, P. R. China}
\email{majiming@fudan.edu.cn}

\author{Baohua Xie}
\address{School of Mathematics, Hunan University, Changsha, 410082, China
}

\email{xiexbh@hnu.edu.cn}



\date{May 21, 2022}

\thanks{Jiming Ma was partially supported by NSFC (No.12171092). Baohua Xie was supported by NSFC (No.11871202) and  Hunan Provincial Natural Science Foundation of China (No.2018JJ3024).}

\begin{document}

\maketitle

\begin{abstract}  We show the manifolds at infinity of the complex hyperbolic triangle groups  $\Delta_{3,4,4;\infty}$ and $\Delta_{3,4,6;\infty}$,  are one-cusped hyperbolic  3-manifolds $m038$ and $s090$ in the Snappy  Census respectively.  That is, these two manifolds admit spherical CR uniformizations.

 Moreover, these two hyperbolic  3-manifolds above can be obtained by Dehn surgeries on the first cusp of the  two-cusped hyperbolic  3-manifold $m295$  in the Snappy  Census with slopes $2$ and $4$ respectively.  In general, the main result in this paper  allow  us to conjecture that  the manifold at infinity of the complex hyperbolic triangle group  $\Delta_{3,4,n;\infty}$ is the one-cusped hyperbolic  3-manifold  obtained by Dehn surgery on the first cusp of  $m295$ with slope  $n-2$.

 \end{abstract}

\section{Introduction}

\subsection{Main results}
Thurston's work on 3-manifolds has shown that geometry has an important role to play in the study of topology of 3-manifolds.
There is a very close relationship between the topological properties of 3-manifolds and the existence of geometric structures on them.
A spherical CR-structure on a smooth 3-manifold $M$ is a maximal collection of distinguished charts modeled on the boundary $\partial \mathbf{H}^2_{\mathbb C}$
of the complex hyperbolic space $\mathbf{H}^2_{\mathbb C}$, where coordinates changes are restrictions of transformations from  $\mathbf{PU}(2,1)$.
 In other words, a {\it spherical CR-structure} is a $(G,X)$-structure with $G=\mathbf{PU}(2,1), X=S^3$. In contrast to the results on other geometric structures carried on 3-manifolds, there are relatively few examples known about spherical CR-structures.

 In general, it is very difficult to determine whether a 3-manifold admits a spherical CR-structure or not. Some of the first examples were given by
 Burns-Shnider \cite{BS:1976}. 3-manifolds with $Nil^{3}$-geometry naturally admit such structures, but by Goldman \cite{Goldman:1983}, any closed 3-manifold with Euclidean or $Sol^3$-geometry
 does not admit such structures.

We are interested in an important class of spherical CR-structures, called uniformizable spherical CR-structures. A spherical CR-structure on a 3-manifold $M$  is {\it uniformizable} if it is
obtained as $M=\Gamma\backslash \Omega_{\Gamma}$, where  $\Omega_{\Gamma}\subset \partial \mathbf{H}^2_{\mathbb C}$ is the set of discontinuity of discrete subgroup  $\Gamma$ acting on $\partial \mathbf{H}^2_{\mathbb C}=S^3$.
Constructing discrete subgroups of $\mathbf{PU}(2,1)$  can be used to constructed spherical CR-structures on 3-manifolds.

Thus, the study of the geometry of discrete subgroups of  $\mathbf{PU}(2,1)$ is crucial to the understanding of uniformizable spherical CR-structures.
Complex hyperbolic triangle groups provide  rich examples of such discrete subgroups. As far as we know, almost all known examples of uniformizable spherical CR-structures are related to complex hyperbolic triangle groups.

Let $\Delta_{p,q,r}$ be the abstract $(p,q,r)$ reflection triangle group with the presentation
$$\langle \sigma_1, \sigma_2, \sigma_3 | \sigma^2_1=\sigma^2_2=\sigma^2_3=(\sigma_2 \sigma_3)^p=(\sigma_3 \sigma_1)^q=(\sigma_1 \sigma_2)^r=id \rangle,$$
where $p,q,r$ are positive integers or $\infty$ satisfying $1/p+1/q+1/r<1$. For simplicity, we assume that $p \leq q \leq r$. If $p,q$ or $r$ equals $\infty$, then
the corresponding relation does not appear.
A \emph{complex hyperbolic $(p,q,r)$ triangle group} is a representation of $\Delta_{p,q,r}$ into $\mathbf{PU}(2,1)$
where the generators fix complex lines. It is well known  that the space of $(p,q,r)$-complex reflection triangle groups has real dimension one if $3 \leq p \leq q \leq r$. Sometimes, we denote the  representation of the triangle group $\Delta_{p,q,r}$ into $\mathbf{PU}(2,1)$ such that $\sigma_1 \sigma_3\sigma_2 \sigma_3$ of order $n$ by $\Delta_{p,q,r;n}$.

Richard Schwartz has conjectured the necessary and sufficient condition for a complex hyperbolic $(p,q,r)$ triangle group $\langle I_1,I_2,I_3\rangle < \mathbf{PU}(2,1)$ to be a discrete and faithful  representation of $\Delta_{p,q,r}$ \cite{schwartz-icm}. Schwartz's conjecture has been proved in a few cases.

We now provide a brief historical overview, before discussing our results.
 The study of complex hyperbolic triangle groups began with Goldman and Parker in \cite{GoPa}.
They considered complex hyperbolic ideal triangle groups, i.e. the case $p=q=r=\infty$, and obtained the result:

\begin{thm}[Goldman and Parker \cite{GoPa}]
Let $\Gamma=\Gamma_{t}=\langle I_1, I_2, I_3 \rangle < \mathbf{PU}(2,1)$ be a complex hyperbolic ideal triangle group, it is parameterized by $t \in [0, \infty)$. Let $t_{1}=\sqrt{105/3}$ and $t_{2}= \sqrt{125/3}$.
If $t > t_{2}$ then $\Gamma_{t}$ is not a discrete embedding of the  $(\infty,\infty,\infty)$ triangle group. If $t < t_{1}$ then $\Gamma_{t}$ is a discrete embedding of the  $(\infty,\infty,\infty)$ triangle group.
\end{thm}

They conjectured that a complex hyperbolic ideal triangle group $\Gamma_{t}=\langle I_1, I_2, I_3 \rangle$ is discrete and faithful if and only if $I_1 I_2 I_3$ is not elliptic, that is, if and only if  $t \leq t_{2}$ in the above parameters.

Schwartz proved the Goldman-Parker conjecture in \cite{Schwartz:2001ann} (see a better proof in \cite{schwartz:2006}).
\begin{thm}[Schwartz \cite{Schwartz:2001ann}]
Let $\Gamma=\langle I_1, I_2, I_3 \rangle < \mathbf{PU}(2,1)$ be a complex hyperbolic ideal triangle group.
If $I_1 I_2 I_3$ is not elliptic, then $\Gamma$ is discrete and faithful.
Moreover, if $I_1 I_2 I_3$ is elliptic, then $\Gamma$ is not discrete.
\end{thm}

Furthermore, he analyzed the group when $I_1 I_2 I_3$ is parabolic.
\begin{thm}[Schwartz \cite{Schwartz:2001acta}]
Let $\Gamma=\langle I_1, I_2, I_3 \rangle$ be the complex hyperbolic ideal triangle group with $I_1 I_2 I_3$ being parabolic.
Let $\Gamma'$ be the even subgroup of $\Gamma$. Then the manifold at infinity of the quotient ${\bf H}^2_{\mathbb C}/{\Gamma'}$ is commensurable with the
Whitehead link complement in the 3-sphere.
\end{thm}

Recently, Schwartz's conjecture was shown for complex hyperbolic $(3,3,n)$ triangle groups with positive integer $n\geq 4$ in \cite{ParkerWX:2016} and $n=\infty$ in \cite{ParkerWill:2016}.
\begin{thm}[Parker, Wang and Xie \cite{ParkerWX:2016}, Parker and Will \cite{ParkerWill:2016}]
Let $n \geq 4$, and let $\Gamma=\langle I_1, I_2, I_3 \rangle$ be a complex hyperbolic $(3,3,n)$ triangle group.
Then $\Gamma$ is a discrete and faithful representation of  the  $(3,3,n)$ triangle group if and only if $I_1 I_3 I_2 I_3$ is not elliptic.
\end{thm}

There are some interesting results on complex hyperbolic $(3,3,n)$ triangle groups with $I_1I_3I_2I_3$ being parabolic.
\begin{thm}[Deraux and Falbel \cite{DerauxF:2015}, Deraux \cite{Deraux:2015} and Acosta \cite{Acosta:2019}]
 Let $4 \leq n \leq +\infty $, and let $\Gamma=\langle I_1, I_2, I_3 \rangle$ be a complex hyperbolic $(3,3,n)$ triangle group with $I_1I_3I_2I_3$ being parabolic.
Let $\Gamma'$ be the even subgroup of $\Gamma$. Then the manifolds at infinity of  the quotient ${\bf H}^2_{\mathbb C}/{\Gamma'}$ is a Dehn surgery on the one of the cusps of the Whitehead link complement with slope $n-2$.
\end{thm}

Note that our choice of the  meridian-longitude systems of the  Whitehead link complement in Theorem 1.5 is different from that in \cite{Acosta:2019}, the  meridian-longitude systems chosen here seems more coherent for the 3-manifold at infinity of the complex hyperbolic triangle group  $\Delta_{3,4,n; \infty}$ below.

These deformations provide an attractive problem, because they furnish some of the simplest interesting examples in the still mysterious subject of complex hyperbolic deformations. While some progress has been made in understanding these examples, there is still a lot unknown about them.

The main purpose of this paper is to study the geometry of triangle groups $\Delta_{3,4,\infty;4}$ and $\Delta_{3,4,\infty;6}$.
Thompson showed  \cite{Thompson:2010} that $\Delta_{3,4,\infty;4}$ and $\Delta_{3,4,\infty;6}$ are arithmetic subgroups of  $\mathbf{PU}(2,1)$, thus they are discrete. We will construct the Ford domains for these two
 groups, this will provide another proof of the discreteness of such two groups. Furthermore, we will identify the manifolds at infinity for them.
It is well-known \cite{kpt, Thompson:2010} that $\Delta_{p,q,r; n}$ and $\Delta_{p,q,n; r}$ are isomorphic. So we often write this family of groups as $\Delta_{3,4,\infty; n}$ or $\Delta_{3,4,n; \infty}$ for convenience.

Our main results are
\begin{thm} \label{thm:main}
Let $\Gamma=\langle I_1, I_2, I_3 \rangle$ be the complex hyperbolic  triangle group $\Delta_{3,4,\infty;n}$ with $I_1I_3I_2I_3$ of order $n$.
\begin{itemize}
  \item If $n=4$, then the manifold at infinity of the even subgroup $\langle I_1I_2,I_2I_3\rangle$ of $\Gamma$ is  the one-cusped hyperbolic 3-manifold $m038$ in the Snappy  Census.
  \item If $n=6$, then the manifold at infinity of the even subgroup $\langle I_1I_2,I_2I_3\rangle$ of $\Gamma$ is  the one-cusped hyperbolic  3-manifold $s090$ in the Snappy  Census.
\end{itemize}
\end{thm}

 The general ideas of the proof as follows. First,  let $\Gamma$ be one of  $\Delta_{3,4,\infty;n}$, for $n=4$ or $6$. We construct a Ford domain $D$ for the discrete group $\Gamma$ and analysis the combinatorics of the faces of this domain. The ideal boundary
  $\partial_{\infty}D$ of $D$, that is  $\partial_{\infty}D=D \cap \partial \mathbf{H}^2_{\mathbb C}$, is crucial to get the manifold at infinity of the  group $\Gamma$. $\partial_{\infty}D$ is the complement of a topological solid cylinder in $\partial{\mathbf{H}^2_{\mathbb C}}-\{q_{\infty}\}$, the boundary of $\partial_{\infty}D$ is an infinite annulus with  polyhedral structure induced by the  polyhedral structure of $D$.
   The combinatorial structure of $\partial_{\infty}D$ can be gathered from the boundary 2-faces of $\partial_{\infty}D$. Unlike the case of the spherical CR uniformization of the figure eight knot complement in \cite{DerauxF:2015}, the combinatorial structure of $\partial_{\infty}D$ is quite different from the Ford domain of the real hyperbolic structure of the figure eight knot complement.
In fact, the combinatorial structure of $\partial_{\infty}D$ for the discrete group $\Delta_{3,4,\infty;n}(n=4,6)$ is more complicated than that in the case of $\Delta_{3,3,n;\infty}(n=4,5)$, since one isometric sphere will contribute two or more boundary 2-faces on $\partial_{\infty}D$.
Then we produce a 2-dimensional picture of the boundary of  $\partial_{\infty}D$, which determines a canonical 2-spine $S$ of our 3-manifold at infinity of ${\bf H}^2_{\mathbb C}/{\Gamma}$.
 After that, we will determine the topology of the manifold at infinity, which is a quotient space of $\partial_{\infty}D$.  From the combinatorial description of $\partial_{\infty}D$, we can calculate the fundamental group of the manifold.  The end result is that the manifolds at infinity will be identified with the hyperbolic 3-manifolds $m038$ and $s090$ respectively \cite{CullerDunfield:2014}.

\subsection{Discussion of related topics}

We observe that the combinatorial structure of the Ford domain of $\Delta_{3,4,\infty; \infty}$ is not too difficult. But the ideal boundary of
this Ford domain in Heisenberg group does not give a horotube structure as in \cite{Deraux:2016,ParkerWill:2016}. Due to the different topology structure, we use more sophisticated methods
to show that

\begin{thm}[Ma and Xie \cite{MaX:2020}]\label{thm:34inftyinfty} The manifold at infinity of the even subgroup of the complex triangle group $\Delta_{3,4,\infty;\infty}$ is the two-cusped hyperbolic 3-manifold $m295$ in the Snappy  Census.
\end{thm}

\begin{thm}[Ma and Xie \cite{MaX:2020}]\label{thm:3inftyinftyinfty}The manifold at infinity of the even subgroup of the complex triangle group $\Delta_{3,\infty,\infty;\infty}$ is the simplest chain link with three components in $S^3$.
\end{thm}

The complement of the  simplest chain link with three components in $S^3$ is  the so called "magic" 3-manifold, see
\cite{MartelliP:2006}.

 It is very possible to show that infinitely many hyperbolic 3-manifolds via Dehn surgeries on   the first cusp of the hyperbolic 3-manifold $m295$  admit spherical CR uniformizations, by   applying M. Acosta's CR Dehn surgery theory \cite{Acosta:2019} and deforming the Ford domain of $\Delta_{3,4,\infty; \infty}$ in a one parameter family.

\begin{conj}\label{conj:34ninfty}The manifold at infinity of the even subgroup of the complex triangle group $\Delta_{3,4,\infty;n}$ is the hyperbolic 3-manifold obtained via the Dehn surgery of
 $m295$ on the first cusp with slope $n-2$.
\end{conj}

We now give some remarks on Theorems \ref{thm:34inftyinfty}, \ref{thm:3inftyinftyinfty}, and Conjecture \ref{conj:34ninfty}. Let $N$ be the simplest chain link with three components in $S^3$ \cite{MartelliP:2006}, which is a hyperbolic link. We use the meridian-longitude systems of the cusps of $N$ as in \cite{MartelliP:2006}, which is different from  the meridian-longitude systems  in Snappy.   Martelli-Petronio classified all the non-hyperbolic Dehn fillings of $N$ \cite{MartelliP:2006}, and got lots of information about hyperbolic Dehn fillings of $N$. Since $N$ has three cusps, we denote by $N(\alpha)$ the filled 3-manifold with two-cusps where the filling slope is $\alpha$ on the first cusp of $N$. Similarly, $N(\alpha, \beta)$ is a one-cusped 3-manifold with filling slopes  $\alpha$ and $\beta$ on the first two cusps of $N$, and $N(\alpha, \beta, \gamma)$ is a closed 3-manifold.

 1. The Dehn filling 3-manifold $N(1)$ is the Whitehead link complement in $S^3$ \cite{MartelliP:2006}, now $N(1)$ has two cusps, which is the manifold at infinity of $\Delta_{3,3,\infty; \infty}$ \cite{ParkerWill:2016}. The manifold $N(1,2)$ is the figure eight  knot complement in the 3-sphere \cite{MartelliP:2006},  which is the manifold at infinity of the index two even subgroup of $\Delta_{3,3,4; \infty}$ \cite{DerauxF:2015}, and  the manifold $N(1,2, -2)$ is the Seifert 3-manifold $(S^2,(3,1),(3,1),(4,1),-1)$. Moreover, $N(1,n-2)$ is a hyperbolic 3-manifold with one cusp for $n \geq 4$,  which is the manifold at infinity of the index two even subgroup of  $\Delta_{3,3,n; \infty}$ \cite{Acosta:2019}.   The manifold $N(1,n-2, -2)$ is the Seifert 3-manifold $\left(S^2,(3,1),(3,1),(n,1),-1\right)$.

2. The manifold $N(2)$ is $m295$ in the Snappy  Census \cite{CullerDunfield:2014}, which is also the link complement $9^2_{50}$ in Rolfsen's list  \cite{Rolfsen}.  The filled 3-manifold $N(2,2)$ is the manifold $m038$, which is the manifold at infinity  of the index two even subgroup of  $\Delta_{3,4,4; \infty}$ as in Theorem \ref{thm:main}. The manifold  $N(2,2,-2)$ is the Seifert 3-manifold $\left(S^2,(3,1),(4,1),(4,1),-1\right)$.
 The filled 3-manifold $N(2,4)$ is the manifold $s090$, which is the manifold at infinity  of the index two even subgroup of  $\Delta_{3,4,6; \infty}$ as in Theorem \ref{thm:main}. We also have  the manifold $N(2,4,-2)$ is the Seifert 3-manifold $\left(S^2,(3,1),(4,1),(6,1),-1\right)$.

 3. From above, it is  natural to  propose the Conjecture  \ref{conj:34ninfty}. Moreover, it should be true that the manifold at infinity of the index two even subgroup of $\Delta_{3,n,\infty; \infty}$ is the two-cusped 3-manifold $N(n-2)$, and the manifold at infinity of the index two even subgroup $\Delta_{3,n,m; \infty}$  is the one-cusped 3-manifold $N(n-2, m-2)$ if Schwartz's conjecture is true.

The paper is organized as follows. In Section 2  we give well known background
material. Section 3 contains the matrix representation of the complex hyperbolic triangle group $\Delta_{3,4,n;\infty}$  in $\mathbf{SU}(2,1)$ for $n=4,6$. Section 4 is devoted to the description of the isometric spheres that bound the Ford domain for the complex hyperbolic triangle group $\Delta_{3,4,4;\infty}$. We also give some examples of calculation in this section. In Section 5, we study combinatorial structure of the ideal boundary of the Ford domain of  the complex hyperbolic triangle group $\Delta_{3,4,4;\infty}$ and get the hyperbolic 3-manifold at infinity.  In Section 6, we study the geometry of  the complex hyperbolic triangle group $\Delta_{3,4,6;\infty}$  similarly but omit some details of the computation.

\section{Background}\label{sec-back}
In this section, we present some preliminaries  on complex hyperbolic geometry.
For more details, see \cite{Go}.

\subsection{Complex hyperbolic plane and Siegel domain}
Let ${\mathbb C}^{2,1}$  denote the vector space ${\mathbb C}^{3}$ equipped with the Hermitian
form
$$\langle {\bf{z}}, {\bf{w}} \rangle =z_1 \overline{w}_3+z_2 \overline{w}_2+z_3 \overline{w}_1$$
of signature $(2,1)$, where ${\bf z}=(z_1,z_2,z_3)^T$ and ${\bf w}=(w_1,w_2,w_3)^T$ are vectors in ${\mathbb C}^3$.
The Hermitian form divides ${\mathbb C}^{2,1}$ into three parts $V_{-}, V_{0}$ and $V_{+}$, which are

\begin{eqnarray*}
  V_{-} &=& \{{\bf z}\in {\mathbb C}^3-\{0\} : \langle {\bf z}, {\bf z} \rangle <0 \}, \\
  V_{0} &=& \{{\bf z}\in {\mathbb C}^3-\{0\} : \langle {\bf z}, {\bf z} \rangle =0 \}, \\
  V_{+} &=& \{{\bf z}\in {\mathbb C}^3-\{0\} : \langle {\bf z}, {\bf z} \rangle >0 \}.
\end{eqnarray*}

Let $$ \mathbb{P}: {\mathbb C}^{3}-\{0\}\rightarrow {\mathbb C}{\mathbb P}^2$$ be  the canonical projection onto the  complex projective space.
Then the {\it complex hyperbolic plane} ${\bf H}^2_{\mathbb C}$ is the image of $V_{-}$ in ${\mathbb C}{\mathbb P}^2$
by the  map ${\mathbb P}$  and its {\it ideal boundary}, or {\it  boundary at infinity}, is  the image of $V_{0}$ in
 ${\mathbb C}{\mathbb P}^2$, we denote it by $\partial {\bf H}^2_{\mathbb C}$.

The standard lift $(z_1,z_2,1)^T$ of $z=(z_1,z_2)\in \mathbb {C}^{2}$ is negative if and only if
$$z_1+|z_2|^2+\overline{z}_1=2{\rm Re}(z_1)+|z_2|^2<0.$$  Thus $\mathbb{P}(V_{-})$ is a paraboloid in ${\mathbb C}^{2}$, called the {\it Siegel domain}.
Its boundary $\mathbb{P}(V_{0})$ satisfies
$$2{\rm Re}(z_1)+|z_2|^2=0.$$

The complex hyperbolic space is parameterized in {\it horospherical coordinates}  by $\mathbb{C}\times \mathbb{R}\times \mathbb{R}^{+}$:
$$(z,t,u)\rightarrow \left(\begin{matrix}\frac{-|z|^2-u+it}{2}\\ z\\1\end{matrix}\right).$$ The standard lift of the point at infinity is
$$q_{\infty}=\left(\begin{matrix}1\\ 0\\0\end{matrix}\right).$$  Then $\mathbb{P}(V_{0})=\{\mathbb{C}\times \mathbb{R}\times \{0\}\}\cup \{q_{\infty}\}$.
Therefore, the Siegel domain has an  analogue construction of the upper half space model for the  real hyperbolic space $H_{\mathbb{R}}^{n}$.

If
$$\mathbf{p}=\left(\begin{matrix}
p_1\\ p_2\\p_3\end{matrix}\right),\quad \mathbf{q}=\left(\begin{matrix}
q_1\\ q_2\\q_3\end{matrix}\right)$$ are lifts of $p,q$ in $\mathbf{H}^2_{\mathbb C}$, then the {\it Hermitian cross product} of $p$ and $q$  is defined by

$$\mathbf{p}\boxtimes \mathbf{q}=\left(\begin{matrix}
\overline{p_1q_2-p_2q_1}\\ \overline{p_3q_1-p_1q_3}\\ \overline{p_2q_3-p_3q_2}\end{matrix}\right).$$
This vector is orthogonal to $\mathbf{p}$ and $\mathbf{q}$  with  respect to the Hermitian form  $\langle \cdot,\cdot\rangle$.
It is a Hermitian version of the Euclidean cross product.

The {\it Bergman metric} $\rho$ on ${\bf H}^2_{\mathbb C}$ is given by the following
$$
  \cosh^2 \left(\frac{\rho(z,w)}{2}\right) = \frac{\langle \mathbf{z}, \mathbf{w} \rangle \langle \mathbf{w},
   \mathbf{z} \rangle}{\langle \mathbf{z}, \mathbf{z} \rangle \langle \mathbf{w}, \mathbf{w} \rangle},
$$
where $ \mathbf{z}, \mathbf{w}\in V_{-} $ are the lifts of $z,w$ respectively. Note that this definition is
 independent of the choices of lifts.

\subsection{The isometries} The complex hyperbolic plane is a K\"{a}hler manifold of constant holomorphic sectional curvature $-1$.
We denote by $\mathbf{U}(2,1)$ the Lie group of $\langle \cdot,\cdot\rangle$ preserving complex linear
 transformations and by $\mathbf{PU}(2,1)$ the group modulo scalar matrices. The group of holomorphic
  isometries of ${\bf H}^2_{\mathbb C}$ is exactly $\mathbf{PU}(2,1)$. It is sometimes convenient to work with
 $\mathbf{SU}(2,1)$, which is a 3-fold cover of  $\mathbf{PU}(2,1)$.

The full isometry group of ${\bf H}^2_{\mathbb C}$ is given by
$$\widehat{\mathbf{PU}(2,1)}=\langle \mathbf{PU}(2,1),\iota\rangle,$$
where $\iota$ is given on the level of homogeneous coordinates by complex conjugate ${\bf z}=(z_1,z_2,z_3)^T  \mapsto {\bf \overline{z}}=(\overline{z}_1,\overline{z}_2,\overline{z}_3)^T$

Elements of $\mathbf{SU}(2,1)$ fall into three types, according to the number and types of the fixed points of the corresponding
isometry. Namely, an isometry is {\it loxodromic} (resp. {\it parabolic}) if it has exactly two fixed points (resp. one fixed point)
on $\partial {\bf H}^2_{\mathbb C}$. It is called {\it elliptic}  when it has (at least) one fixed point inside ${\bf H}^2_{\mathbb C}$.
 An elliptic $A\in \mathbf{SU}(2,1)$ is called {\it regular elliptic} whenever it has three distinct eigenvalues, and {\it special elliptic} if
 it has a repeated eigenvalue.

Suppose that  a non-identity element $T\in \mathbf{SU}(2,1)$ has trace equal to 3.  Then all eigenvalues of $T$ equal 1, that is $T$ is {\it unipotent}.
If $A\in\mathbf{SU}(2,1)$ fixes $q_{\infty}$, then it is upper triangular. We now examine the subgroup of  $\mathbf{SU}(2,1)$ fixing
$q_{\infty}$. Consider the map $T$ from $\partial {\bf H}^2_{\mathbb C}-\{q_{\infty}\}=\mathbb{C}\times \mathbb{R}$ to $\mathbf{GL}(3,\mathbb{C})$ given by
$$T_{(z,t)}=\left(\begin{matrix}
1 & -\overline{z}& \frac{-|z|^{2}+it}{2} \\ 0 & 1 & z \\ 0 & 0 & 1 \end{matrix}\right).$$
It is easy to find that this map fixes $q_{\infty}$ and sends the origin in $\mathbb{C}\times \mathbb{R}$ to the point $(z,t)$.

Moreover, composition of such elements gives $\partial {\bf H}^2_{\mathbb C}-\{q_{\infty}\}$ the structure of the Heisenberg group
$$(z_1,t_1)\cdot(z_2,t_2)=\left(z_1+z_2,t_1+t_2+2{\rm Im}(z_1\overline{z}_2)\right)$$ and $T_{(z,t)}$ acts as left {\it Heisenberg translation}
on $\partial {\bf H}^2_{\mathbb C}-\{q_{\infty}\}$. A Heisenberg translation by $(0,t)$ is called a {\it vertical translation} by $t$.

The full stabilizer of $q_{\infty}$ is generated by the above unipotent group, together with the isometries of the forms
\begin{equation}
\left(\begin{matrix}
1 & 0& 0 \\ 0 & e^{i\theta} & 0 \\ 0 & 0 & 1 \end{matrix}\right) \quad  and   \quad
\left(\begin{matrix}
\lambda & 0 & 0 \\ 0 & 1 & 0 \\ 0 & 0 & 1/\lambda \end{matrix}\right),
\end{equation}
where $\theta,\lambda\in \mathbb{R}$ and $\lambda \neq 0$.  The first acts on $\partial {\bf H}^2_{\mathbb C}-\{q_{\infty}\}=\mathbb{C}\times \mathbb{R}$ as a rotation
with vertical axis:
$$(z,t)\mapsto (e^{i\theta}z,t),$$  whereas the second one acts as $$(z,t)\mapsto (\lambda z,\lambda^2 t).$$

The Heisenberg norm of $(z,t)\in \partial {\bf H}^2_{\mathbb C}-\{q_{\infty}\}$ is given by
 $$\parallel(z,t)\parallel=\left||z|^2+it\right|^{1/2}.$$
This gives rise to a metric, the {\it Cygan metric}, on Heisenberg group by
$$d((z_1,t_1),(z_2,t_2))=\parallel (z_1,t_1)^{-1}\cdot(z_2,t_2)\parallel.$$
The {\it Cygan sphere} with center $(z_0,t_0)$  and radius $r$ has equation
$$d\left((z,t),(z_0,t_0)\right)=\left||z-z_0|^{2}+i(t-t_0+2{\rm Im}(z\overline{z}_0))\right|=r^2.$$
The Cygan metric can be extended to the metric to points $p$ and $q$ in ${\bf H}^2_{\mathbb C}$  with horospherical coordinates $(z_1,t_1,u_1)$
and $(z_2,t_2,u_2)$ by writing
$$d\left(p,q\right)=\left||z_1-z_2|^{2}+|u_1-u_2|+i(t_1-t_2+2{\rm Im}(z_1\overline{z}_2))\right|^{1/2}.$$

\subsection{Totally geodesic submanifolds and complex reflections}
There are two kinds of totally geodesic submanifolds of real dimension 2 in ${\bf H}^2_{\mathbb C}$: {\it complex lines} in ${\bf H}^2_{\mathbb C}$ are
complex geodesics (represented by ${\bf H}^1_{\mathbb C}\subset {\bf H}^2_{\mathbb C}$) and {\it Lagrangian planes} in ${\bf H}^2_{\mathbb C}$ are totally
real geodesic 2-planes (represented by ${\bf H}^2_{\mathbb R}\subset {\bf H}^2_{\mathbb C}$). Since the Riemannian sectional curvature of the  complex hyperbolic space
 is nonconstant, there are no totally geodesic hyperplanes.

Consider the complex hyperbolic space ${\bf H}^2_{\mathbb C}$  and its boundary $\partial{\bf H}^2_{\mathbb C}$. We define
\emph{$\mathbb{C}$-circles} in $\partial{\bf H}^2_{\mathbb C}$ to be the boundaries of complex geodesics in ${\bf H}^2_{\mathbb C}$. Analogously,
We define \emph{$\mathbb{R}$-circles} in $\partial{\bf H}^2_{\mathbb C}$ to be the boundaries of Lagrangian planes in ${\bf H}^2_{\mathbb C}$.

Let $L$ be a complex line and $\partial L$ be its trace on boundary $\partial{\bf H}^2_{\mathbb C}$. A {\it polar vector} of $L$ (or $\partial L$) is the unique vector (up to scalar multiplication) perpendicular to this complex line with
respect to the Hermitian form. A polar vector  belongs to  $V_{+}$ and each vector in $V_{+}$ corresponds to a complex line or a $\mathbb{C}$-circle.
Moreover, let $L$ be a complex line with polar vector
${\bf c}\in V_{+}$,
then the {\it complex reflection} fixing $L$ is given by
$$
  I_{\bf c}({\bf z}) = -{\bf z}+2\frac{\langle {\bf z}, {\bf c} \rangle}{\langle {\bf c},{\bf c}\rangle}{\bf c}.
$$

In the Heisenberg model, $\mathbb{C}$-circles are either vertical lines or ellipses, whose projection on the $z$-plane are circles.
 Finite $\mathbb{C}$-circles are determined by a center and a radius. They may also be described using polar vectors. A finite $\mathbb{C}$-circles
 with center $(x+yi,z) \in \mathbb{C} \times \mathbb{R}$ and radius $r$ has polar vector
 $$\left(\begin{matrix} \frac{r^2-x^2-y^2+iz}{2}\\ x+y i \\ 1 \end{matrix}\right).$$

\subsection{Bisectors and spinal coordinates}
In order to analyze  2-faces of a Ford polyhedron, we must study the intersections of isometric spheres.
Isometric spheres are special example of bisectors. In this subsection, we will describe a convenient set of
coordinates for bisector intersections, deduced from slice decomposition.
\begin{defn} Given two distinct points $p_0$  and $p_1$ in ${\bf H}^2_{\mathbb C}$ with the same norm (e.g. one could
take $\langle \mathbf{p}_0,\mathbf{p}_0\rangle=\langle \mathbf{p}_1,\mathbf{p}_1\rangle= -1$), the \emph{bisector} $\mathcal{B}(p_0,p_1)$ is the projectivization of the set of negative vectors
$x$ with
$$|\langle x,\mathbf{p_0}\rangle|=|\langle x,\mathbf{p_1}\rangle|.$$
\end{defn}

The  {\it spinal sphere} of the bisector $\mathcal{B}(p_0,p_1)$ is the intersection of $\partial {\bf H}^2_{\mathbb C}$ with the closure of $\mathcal{B}(p_0,p_1)$ in $\overline{{\bf H}^2_{\mathbb C}}= {\bf H}^2_{\mathbb C}\cup \partial { {\bf H}^2_{\mathbb C}}$. The bisector $\mathcal{B}(p_0,p_1)$ is a topological 3-ball, and its spinal sphere is a 2-sphere.
The  {\it complex spine} of $\mathcal{B}(p_0,p_1)$ is the complex line through the two points $p_0$ and $p_1$. The {\it real spine} of $\mathcal{B}(p_0,p_1)$
is the intersection of the complex spine with the bisector itself, which is a (real) geodesic; it is the locus of points inside the complex spine which are equidistant from $p_0$ and $p_1$.
Bisectors are not totally geodesic, but they have a very nice foliation by two different families of totally geodesic submanifolds. Mostow \cite{Mostow:1980}  showed that a bisector is the preimage of the real
spine under the orthogonal projection onto the complex spine. The fibres of this projection are complex lines called the {\it complex slices} of the bisector. Goldman \cite{Go} showed that a bisector is
the union of all  Lagrangian planes containing the real spine. Such Lagrangian planes are called the {\it real slices} of the bisector.

From the detailed analysis in \cite{Go}, we know that the intersection of two bisectors is usually not totally geodesic and can be somewhat complicated. In  this paper, we shall only consider the intersection of
coequidistant bisectors, i.e. bisectors equidistant from a common point.   When $p,q$ and $r$ are not in a common complex line, that is,  their lifts are linearly independent in $\mathbb {C}^{2,1}$, then the locus $\mathcal{B}(p,q,r)$ of points in $ {\bf H}^2_{\mathbb C}$
equidistant to  $p,q$ and $r$ is a smooth disk that is not totally geodesic, and is often called a \emph{Giraud disk}.  The following property is crucial when studying fundamental domain.

\begin{prop}[Giraud]
 If $p,q$ and $r$ are not in a common complex line, then $\mathcal{B}(p,q,r)$ is contained in precisely three bisectors, namely $\mathcal{B}(p,q),  \mathcal{B}(q,r)$ and  $\mathcal{B}(p,r)$.
\end{prop}

Note that checking whether an isometry maps  a Giraud disk to another is equivalent to checking that corresponding triple of  points are mapped to each other.

In order to study Giraud  disks, we will use spinal coordinates. The complex slices of $\mathcal{B}(p,q)$ are given explicitly by choosing a lift $\mathbf{p}$ (resp. $\mathbf{q}$)
of $p$ (resp. $q$).

When $p,q\in  {\bf H}^2_{\mathbb C}$, we simply choose lifts such that $\langle \mathbf{p},\mathbf{p}\rangle= \langle \mathbf{q},\mathbf{q}\rangle$.
In this paper, we will mainly use these parametrization when  $p,q\in  \partial{\bf H}^2_{\mathbb C}$. In that case, all lifts are null vectors and the condition
 $\langle \mathbf{p},\mathbf{p}\rangle= \langle \mathbf{q},\mathbf{q}\rangle$  is vacuous. Then we choose some fixed lift $\mathbf{p}$
 for  the center of the Ford domain, and we take $ \mathbf{q}=G( \mathbf{p})$ for some $G\in \mathbf{ SU}(2,1)$.

The complex slices of $\mathcal{B}(p,q)$ are obtained as the set of negative lines $(\overline{z}\mathbf{p}-\mathbf{q})^{\bot}$ in ${\bf H}^2_{\mathbb C}$ for some arc of values of $z\in S^1$,  which is determined by requiring that $\langle \overline{z}\mathbf{p}-\mathbf{q},\overline{z}\mathbf{p}-\mathbf{q}\rangle>0$.

Since a point of the bisector is on precisely one complex slice, we can parameterize the Giraud torus $\hat{\mathcal{B}}(p,q,r)$  by $(z_1,z_2)=(e^{it_1},e^{it_2})\in S^1\times S^1 $ via
$$V(z_1,z_2)=(\overline{z}_1\mathbf{p}-\mathbf{q})\boxtimes (\overline{z}_2\mathbf{p}-\mathbf{r})=\mathbf{q}\boxtimes \mathbf{r}+z_1 \mathbf{r}\boxtimes \mathbf{p}+z_2 \mathbf{p}\boxtimes \mathbf{q}.$$

The Giraud disk $\mathcal{B}(p,q,r)$ corresponds to the $(z_1,z_2)\in S^1\times S^1 $  with  $$\langle V(z_1,z_2),V(z_1,z_2)\rangle<0.$$ It follows from the fact the bisectors are covertical that  this region is a topological disk, see \cite{Go}.

The boundary at infinity $\partial\mathcal{B}(p,q,r)$ is a circle, given in spinal coordinates by the equation
 $$\langle V(z_1,z_2),V(z_1,z_2)\rangle=0.$$

 A defining equation for the trace of another bisector $\mathcal{B}(u,v)$ on the Giraud disk $\mathcal{B}(p,q,r)$ can be written in the form
 $$| \langle V(z_1,z_2),u\rangle|=| \langle V(z_1,z_2),v\rangle|,$$  provided that $u$ and $v$ are suitably chosen lifts.
 The expressions $\langle V(z_1,z_2),u\rangle$ and $\langle V(z_1,z_2),v\rangle$ are affine in $z_1, z_2$.

 This triple bisector intersection can be parameterized  fairly explicitly, because one can solve the equation
 $$|\langle V(z_1,z_2),u\rangle|^2=|\langle V(z_1,z_2),v\rangle|^2 $$ for one of the variables $z_1$ or $z_2$ simply by solving a quadratic equation.
  A detailed explanation of how this works can be found in \cite{Deraux:2016,dpp}.

\subsection{Isometric spheres  and Ford domain}
\begin{defn} For any $G\in \mathbf{SU}(2,1)$ that does not fix $q_{\infty}$, the \emph{isometric sphere} of $G$, we denote it  by $\mathcal{I}(G)$, is defined to be
$$\mathcal{I}(G)=\{p\in {\bf H}^2_{\mathbb C} \cup \partial{\bf H}^2_{\mathbb C}: |\langle \mathbf{p},q_{\infty}\rangle|=|\langle \mathbf{p},G^{-1}(q_{\infty})\rangle|\},$$
where $ \mathbf{p}$ is the standard lift of $p\in {\bf H}^2_{\mathbb C} \cup \partial{\bf H}^2_{\mathbb C}$.
\end{defn}

The {\it interior} of $\mathcal{I}(G)$ is the component of its complement in  ${\bf H}^2_{\mathbb C} \cup \partial{\bf H}^2_{\mathbb C}$ that does not contain $q_{\infty}$, namely,
$$\{p\in {\bf H}^2_{\mathbb C} \cup \partial{\bf H}^2_{\mathbb C}: |\langle \mathbf{p},q_{\infty}\rangle|>|\langle \mathbf{p},G^{-1}(q_{\infty})\rangle|\}.$$
The {\it exterior} of $\mathcal{I}(G)$  is the component that contains the point at infinity $q_{\infty}$.

Suppose that $G\in \mathbf{SU}(2,1)$ written in $3\times 3$  complex matrices $(g_{ij})_{1\leq i,j\leq 3}$ does not fix $q_{\infty}$.
Then $G^{-1}(q_{\infty})=(\overline{g}_{33},\overline{g}_{32},\overline{g}_{31})^{T}$ and $g_{31}\neq 0$.  Meanwhile, the horospherical coordinate of $G^{-1}(q_{\infty})$ is
$$G^{-1}(q_{\infty})=\left(\frac{\overline{g}_{32}}{\overline{g}_{31}}, 2{\rm Im}\left(\frac{\overline{g}_{33}}{\overline{g}_{31}}\right)\right).$$

We can also describe the  isometric sphere of $G$ by using the Cygan metric.  For concreteness, the  isometric sphere  $\mathcal{I}(G)$ is
the Cygan sphere in ${\bf H}^2_{\mathbb C} \cup \partial{\bf H}^2_{\mathbb C}$ with center $G^{-1}(q_{\infty})$ and radius $r_G=\sqrt{\frac{2}{|g_{31}|}}$.

Furthermore, we have

\begin{prop} Let $G\in \mathbf{SU}(2,1)$  be an isometry of ${\bf H}^2_{\mathbb C}$ not fixing $q_{\infty}$.
\begin{itemize}
  \item The transformation $G$  maps $\mathcal{I}(G)$ to $\mathcal{I}(G^{-1})$, and the interior of $\mathcal{I}(G)$ to the exterior of $\mathcal{I}(G^{-1})$.
  \item For any $A\in\mathbf{SU}(2,1)$ fixing $q_{\infty}$ and such that the corresponding eigenvalues has unit modulus, we have $\mathcal{I}(G)=\mathcal{I}(AG)$.
\end{itemize}
\end{prop}

\begin{defn}The \emph{Ford domain} $D_{\Gamma}$ for a discrete group $\Gamma < \mathbf{PU}(2,1)$ centered at $q_{\infty}$ is
the intersection of the (closures of the) exteriors of all isometric spheres for elements of $\Gamma$ not fixing $q_{\infty}$. That is,
$$D_{\Gamma}=\{p\in {\bf H}^2_{\mathbb C} \cup \partial{\bf H}^2_{\mathbb C}: |\langle \mathbf{p},q_{\infty}\rangle|\leq|\langle \mathbf{p},G^{-1}(q_{\infty})\rangle|
\ \forall G\in \Gamma \ \mbox{with} \ G(q_{\infty})\neq q_{\infty} \}.$$
\end{defn}

From the definition, one can see that  isometric spheres form the boundary of the Ford domain.
When $q_{\infty}$ is either in the domain  of discontinuity or is a parabolic fixed point, the Ford  domain  is preserved by $\Gamma_{\infty}$, the stabilizer of
$q_{\infty}$ in $\Gamma$.  In this case, $D_{\Gamma}$ is only a fundamental domain modulo the action of $\Gamma_{\infty}$.  In other words, the fundamental domain for
$\Gamma$ is the intersection of the Ford domain with a fundamental domain for $\Gamma_{\infty}$. Facets of codimension one, two, three and four in $D_{\Gamma}$ will be called {\it sides}, {\it ridges}, {\it edges} and {\it vertices}, respectively. Moreover, a  {\it bounded ridge} is a ridge which does not intersect  $\partial {\bf H}^2_{\mathbb C}$, and if the intersection of a ridge $r$ and $\partial {\bf H}^2_{\mathbb C}$ is non-empty, then
 $r$ is an {\it infinite ridge}.

It is usually very hard to determine $D_{\Gamma}$ because one should check infinitely many inequalities.
Therefore a general method will be to guess the Ford polyhedron and check  it using the Poincar\'e polyhedron theorem.
The basic idea is that the sides of $D_{\Gamma}$ should be paired by isometries, and the images of $D_{\Gamma}$
under these so-called side-pairing maps should give a local tiling of ${\bf H}^2_{\mathbb C}$.  If they do (and if
 the quotient of $D_{\Gamma}$ by the identification given by the side-pairing maps is complete), then the Poincar\'{e} polyhedron
 theorem implies that the images of $D_{\Gamma}$ actually give a global tiling of ${\bf H}^2_{\mathbb C}$.

Once a fundamental domain is obtained, one gets an explicit presentation of $\Gamma$ in terms of the generators given by the side-pairing maps together
with a generating set for the stabilizer $\Gamma_{\infty}$, where the relations corresponding to so-called ridge cycles, which correspond to the local tiling bear each codimension-two faces. For more on the Poincar\'e polyhedron theorem, see \cite{dpp, ParkerWill:2016}.

\section{The representation of the triangle group $\Delta_{3,4,\infty;n}$}\label{sec-gens}
We will explicitly give a matrix representation of $\Delta_{3,4,\infty; n}$ in $\mathbf{SU}(2,1)$ for $n=4,6$. First note that $\Delta_{3,4,\infty;n}$ is   isomorphic to $\Delta_{3,4,n;\infty}$ as subgroups of $\mathbf{SU}(2,1)$ \cite{kpt, Thompson:2010}.
Let $I_1$, $I_2$ and $I_3$ be the three complex reflections of $\Delta_{3,4,\infty;n}$. Suppose that complex reflections $I_1$ and $I_2$
in $\mathbf{SU}(2,1)$ so that $I_1I_2$ is a parabolic element fixing $q_{\infty}$. By using the similar argument in \cite{MaX:2020},
the matrices of $I_1$ and $I_2$ are given by
\begin{equation*}\label{eq-I1-I2}
I_1=\left[\begin{matrix}
-1 & 0& 0 \\ 0 & 1 & 0 \\ 0 & 0 & -1 \end{matrix}\right], \quad
I_2=\left[\begin{matrix}
-1 & -2 & 2 \\ 0 & 1 & -2 \\ 0 & 0 & -1 \end{matrix}\right]
\end{equation*}
and the matrix of $I_3$ is
\begin{equation*}\label{eq-I3}
I_3=\left[\begin{matrix}
-\frac{1}{2}& \frac{5-\sqrt{23}i}{12}& \frac{1}{3} \\ \frac{5+\sqrt{23}i}{8} &0 &\frac{5+\sqrt{23}i}{12} \\ \frac{3}{4} &  \frac{5-\sqrt{23}i}{8} & -\frac{1}{2} \end{matrix}\right], \quad
I_3=\left[\begin{matrix}
-\frac{1}{2}& \frac{5-\sqrt{23}i}{12}& \frac{1}{3} \\ \frac{5+\sqrt{23}i}{8} &0 &\frac{5+\sqrt{23}i}{12} \\ \frac{3}{4} &  \frac{5-\sqrt{23}i}{8} & -\frac{1}{2} \end{matrix}\right]
\end{equation*} for $n=4,6$  respectively.

\medskip

\section{The combinatorics of the Ford domain for the  even subgroup of $\Delta_{3,4,\infty;4}$}\label{section:comb344}
In this section, we will construct the Ford domain for the subgroup of $\Delta_{3,4,\infty;4}$ generated by the elements $I_1I_2$ and $I_2I_3$.  We will also describe the combinatorial structure
of the Ford domain.

\begin{defn} Let $\Gamma_1$ be the even length subgroup of the $\Delta_{3,4,\infty;4}$-triangle group, i.e. the subgroup generated by $I_1I_2$ and $I_2I_3$.
We define $A=I_1I_2, B=I_2I_3$, and write $a,b$ for $A^{-1}, B^{-1}$ respectively.
\end{defn}

One easily checks that

\begin{equation}\label{eq-a-b}
A=\left[\begin{matrix}
1 & 2& -2 \\ 0 & 1 & 2 \\ 0 & 0 & 1 \end{matrix}\right], \quad
B=\left[\begin{matrix}
\frac{1-2i}{2} & \frac{1-i}{2} & \frac{-5-2i}{2}\\ \frac{-1+i}{2}& -1+i & \frac{3+i}{2} \\ -\frac{1}{2} & \frac{-1+i}{2}& \frac{1}{2} \end{matrix}\right].
\end{equation}

First, we give a brief summary of the method  to get the  Ford domain.

A reduced word  consists of $A,B,a,b$ is called an \emph{essential word} in $\Gamma_1$  if the head or the tail of this word is not $A$ or $a$ and the word can not be reduced
 to a shorter one in $\Gamma_1$. For example, $BaB$, $Bab$, $baaB$ are essential words and $aB$, $BA$, $baBA$ are not  essential words. Note that $B^3=id$ in $\Gamma_1$, so
 neither $B^2$ nor  $BAB^2$ is an   essential word, but $b$ and $BAb$ are  essential words

One can find that
\begin{itemize}
  \item the essential words for the group elements of length 1 are  $b,B$;
  \item  there are no  essential words for the group elements of length 2;
  \item the essential words for the group elements of length 3 are  $Bab$, $BaB$, $BAB$, $bAB$  and their inverses;
  \item the essential words for the group elements of length 4 are  $Baab$, $BaaB$, $BAAB$, $bAAB$ and their inverses.
    \end{itemize}

Let $S_n$ be the set of group elements in $\Gamma_1$  that can be expressed as essential words of length at most $n$. The set $S_n$ is symmetric, that is,
if $g\in S_n$, then $g^{-1}$ also belongs to $S_n$.

The polyhedron $D_{S_n}$ will be the intersection of the exteriors of $\mathcal{I}(g)$ for all $g\in S_n$ and all their translates by powers of $A$.  The polyhedron $D_{S_n}$ is also called a  \textsl{partial Ford domain}. Start with
$D_{S_3}$ or $D_{S_4}$ and do the following:

(1) Check if the polyhedron $D_{S_n}$ is the Ford domain by using the Poincar\'e polyhedron theorem.  If it is, we find the Ford domain for $\Gamma_1$  and stop the
procedure. Else go to Step (2).

(2) $D_{S_n}$ does not have side pairings. We will add some essential words of length $n+1$ to $D_{S_n}$ and go to Step (1).

We do not know  a priori that the procedure will stop in finite steps due to Bowditch's result \cite{Bowditch:1993}.   Fortunately, the procedure was  terminated  quickly in our cases.  We find  the set $S^*$ with the  essential words
$$b,B,Bab,BAb,BaB,bAb, BAB,bab,bAB,baB.$$

 We denote the polyhedron $D_{S^*}$ be the intersection of the exteriors of the isometric spheres of this ten elements and all their translates by powers of $A$. The  polyhedron $D_{S^*}$ is our guess for the Ford domain of $\Gamma_1$.
We will show that
\begin{thm} \label{thm:Ford} $D_{S^*}$ is the Ford domain of $\Gamma_1$.
\end{thm}

The tool for the proof of this Theorem \ref{thm:Ford} is the Poincar\'e polyhedron theorem. The key step in the verification of the hypotheses of the Poincar\'e polyhedron theorem will be the determination of the combinatorics of $D_{S^*}$.
We start with the isometric spheres of $b,BAb,bAb,bab,baB$ and their inverses. First, we fix some notations:

\begin{defn} For $k\in \mathbb{Z}$, let $\mathcal{I}_{g}^{k}$ be the isometric sphere of   $A^{k}gA^{-k}$, which is $\mathcal{I}(A^{k}gA^{-k})=A^{k}\mathcal{I}(g)$, where $g\in S^*$  and $\mathcal{I}_{g}^{0}=\mathcal{I}_{g}$. The spinal sphere corresponds to $\mathcal{I}_{g}^{k}$ will be denoted by $\mathcal{S}_{g}^{k}$, where $\mathcal{S}_{g}^{0}=\mathcal{S}_{g}$.
\end{defn}

Note that $$\mathcal{I}(bAb)=\mathcal{I}(ABaBa),\  \mathcal{I}(bab)=\mathcal{I}(aBABA). $$ So the isometric spheres of the following eight group elements
$$b,B,BAb,Bab,bAb,bab,baB,bAB$$  and their conjugates by powers of $A$ will define all the sides of our Ford domain. We summarize the information of these
isometric spheres in  Table \ref{table:centerradius}.

\begin{table}[htbp]
\caption{The centers and radii of the eight spinal spheres.}
\centering
\begin{tabular}{c c c| c c c}
\hline
Spinal sphere & Center & Radius & Spinal sphere& Center & Radius \\ [1 ex]
 \hline
 $\mathcal{S}_{b}$ & $[1,-1,4]$ & $2$ & $\mathcal{S}_{B}$ & $[1,1,0]$ & $2$   \\ [2 ex]
 $\mathcal{S}_{baB}$ & $[\frac{7}{5},\frac{1}{5},4]$ & $\frac{2}{5^{\frac{1}{4}}}$ &  $\mathcal{S}_{bAB}$ &$[\frac{3}{5},\frac{1}{5},-\frac{4}{5}]$ & $\frac{2}{5^{\frac{1}{4}}}$  \\[2 ex]
 $\mathcal{S}_{bAb}$ & $[0,0,4]$ & $\sqrt{2}$ & $\mathcal{S}_{bab}$ & $[2,0,0]$ & $\sqrt{2}$   \\[2 ex]
$\mathcal{S}_{BAb}$  & $[\frac{3}{5},-\frac{1}{5},\frac{24}{5}]$ & $\frac{2}{5^{\frac{1}{4}}}$ & $\mathcal{S}_{Bab}$ & $[\frac{7}{5},-\frac{1}{5},0]$ & $\frac{2}{5^{\frac{1}{4}}}$  \\  [2 ex]
 \hline
\end{tabular}
\label{table:centerradius}
\end{table}

Note that the spinal spheres in Table \ref{table:centerradius}  do not contain the point $q_{\infty}$. So they are bounded sets in $\partial {\bf H}^2_{\mathbb C}-\{q_{\infty}\}$.
For any two spinal spheres $\mathcal{S}_g$ and $\mathcal{S}_h$, we have $A^{k}(\mathcal{S}_g)\cap \mathcal{S}_h=\emptyset$ whenever $k$ is large enough.
It is well known that two spinal spheres intersect if and only if the corresponding bisectors intersect. By some simple calculation, we can provide some rough
estimate listed in Table \ref{table:intersecton}.

\begin{table}[htbp]
\caption{The intersections of $\mathcal{I}_{B}$ and it's neighbouring isometric spheres.}
\centering
\begin{tabular}{c c | c  c}
\hline
Intersection & The value of $k$ & Intersection & The value of $k$ \\ [1 ex]
 \hline
 $\mathcal{I}_{B}\cap \mathcal{I}_{B}^{k}$  & $k=\pm 1$ &$\mathcal{I}_{B}\cap \mathcal{I}_{b}^{k}$ & $k=0, \pm1$    \\ [2 ex]
$\mathcal{I}_{B}\cap \mathcal{I}_{baB}^{k}$ &  $k=0,\pm1$ &$\mathcal{I}_{B}\cap \mathcal{I}_{bAB}^{k}$ & $k=0,\pm1$    \\ [2 ex]
$\mathcal{I}_{B}\cap \mathcal{I}_{BaB}^{k}$ &  $k=0,\pm1 $ &$\mathcal{I}_{B}\cap \mathcal{I}_{BAB}^{k}$ & $k=0,\pm1$    \\ [2 ex]
$\mathcal{I}_{B}\cap \mathcal{I}_{BAb}^{k}$ &  $k=0, \pm1 $ &$\mathcal{I}_{B}\cap \mathcal{I}_{Bab}^{k}$ & $k=0,\pm1$    \\ [2 ex]
\hline
\end{tabular}
\label{table:intersecton}
\end{table}

 The intersection of $\mathcal{I}_{B}$ and  $D_{S^*}$  is a 3-face, that is, a side of the  partial Ford domain $D_{S^*}$. The intersection $\mathcal{I}_{B} \cap D_{S^*}$ is a 3-dimensional polytope, we describe its combinatorics in the following proposition, which is very important for the proof of Theorem \ref{thm:main}.

\begin{prop}\label{prop: combinatorics}
 The combinatorics of all the 2-faces(ridges) on $\mathcal{I}_{B} \cap D_{S^*}$ are listed as follows:
\begin{itemize}
  \item $\mathcal{I}_{B}\cap \mathcal{I}_{baB}^{\pm 1} \cap D_{S^*}$, $\mathcal{I}_{B}\cap \mathcal{I}_{bAB}^{\pm 1} \cap D_{S^*}$, $\mathcal{I}_{B}\cap \mathcal{I}_{BaB}^{-1} \cap D_{S^*}$, $\mathcal{I}_{B}\cap \mathcal{I}_{BaB}^{1} \cap D_{S^*}$, $\mathcal{I}_{B}\cap \mathcal{I}_{BAb}^{ 1} \cap D_{S^*}$ and $\mathcal{I}_{B}\cap \mathcal{I}_{Bab}^{ -1} \cap D_{S^*}$   are empty;

  \item $\mathcal{I}_{B}\cap \mathcal{I}_{BAb}^{-1} \cap D_{S^*}$ and $\mathcal{I}_{B}\cap \mathcal{I}_{Bab}^{1} \cap D_{S^*}$ are combinatorially triangles with one side at infinity; this means, for example, $\mathcal{I}_{B}\cap \mathcal{I}_{Bab}^{1} \cap D_{S^*}$ is a triangle with one ideal boundary in $\partial {\bf H}^2_{\mathbb C}$. See part (f) and (e) of  Figure  \ref{fig:B};

  \item  $\mathcal{I}_{B}\cap \mathcal{I}_{BaB}^{1} \cap D_{S^*}$, $\mathcal{I}_{B}\cap \mathcal{I}_{BAB}^{-1} \cap D_{S^*}$, $\mathcal{I}_{B}\cap \mathcal{I}_{BAb} \cap D_{S^*}$  and $\mathcal{I}_{B}\cap \mathcal{I}_{Bab} \cap D_{S^*}$ are combinatorially quadrangles with one side at infinity.  See part (h),  (g),  (m) and  (l) of  Figure  \ref{fig:B};

  \item $\mathcal{I}_{B}\cap \mathcal{I}_{B}^{\pm 1} \cap D_{S^*}$, $\mathcal{I}_{B}\cap \mathcal{I}_{baB} \cap D_{S^*}$ and $\mathcal{I}_{B}\cap \mathcal{I}_{bAB} \cap D_{S^*}$ are combinatorially pentagons with one side at infinity. See part (c),  (d),  (j) and  (k) of  Figure  \ref{fig:B};

  \item $\mathcal{I}_{B}\cap \mathcal{I}_{b} \cap D_{S^*}$ is a combinatorially dodecagon with no side at infinity.  See part (i) of  Figure  \ref{fig:B};

  \item The combinatorial type of $\mathcal{I}_B\cap  \mathcal{I}_{b}^{1} \cap D_{S^*}$, $\mathcal{I}_B\cap  \mathcal{I}_{b}^{-1} \cap D_{S^*}$, $\mathcal{I}_B\cap  \mathcal{I}_{BaB} \cap D_{S^*}$ and $\mathcal{I}_B\cap  \mathcal{I}_{BAB} \cap D_{S^*}$ are the same. It is a union of a pentagon including a side at infinity and a quadrangle with a common vertex.  See part (a),  (b),  (o) and  (n) of  Figure  \ref{fig:B}.

\end{itemize}

\end{prop}

We summarize  the combinatorics of all 2-faces of the side  $\mathcal{I}_B\cap D_{S^*}$ in Figure  \ref{fig:B}.  See Figures  \ref{fig:bAB}, \ref{fig:bab}, \ref{fig:bAb} and
\ref{fig:Bab} for the combinatorics of all 2-faces of the sides  $\mathcal{I}_{bAB}\cap D_{S^*}$, $\mathcal{I}_{bab}\cap D_{S^*}$, $\mathcal{I}_{bAb}\cap D_{S^*}$ and $\mathcal{I}_{Bab}\cap D_{S^*}$ respectively.

Using similar method in \cite{Deraux:2016, dpp}, we give some examples of calculation, the other routine calculations will be omitted. Throughout the calculation, we denote by $\hat{\mathcal{I}}_{B}$ the extor in the  projective space extending the bisector  $\mathcal{I}_{B}$ (see \cite{Go} for a definition of extor).

\begin{prop} $\mathcal{I}_{B}\cap \mathcal{I}_{baB}^{1}$ is a Giraud disk, which is entirely contained in the interior of the isometric sphere of $b$.
\end{prop}
\begin{proof}The Giraud torus  $\hat{\mathcal{I}}_{B}\cap \hat{\mathcal{I}}_{baB}^{1}$ can be parameterized  by  the vector $$V=(\overline{z}_1p_0-b(p_0))\boxtimes (\overline{z}_2p_0-AbAB(p_0)),$$ with $|z_1|=|z_2|=1$. One can choose the following sample point
\begin{eqnarray*}
X_0&=&\left(\left(\frac{1}{2}+\frac{\sqrt{3}}{2}i\right)p_0-b(p_0)\right)\boxtimes \left(p_0-AbAB(p_0)\right)\\
& = & \left(\frac{5}{4}+\frac{\sqrt{3}}{4}+i\left(-\frac{1}{4}+\frac{\sqrt{3}}{4}\right),-\frac{3}{4}+\frac{\sqrt{3}}{2}+i\left(-\frac{1}{2}+\frac{\sqrt{3}}{4}\right),i\right)
\end{eqnarray*} to show that $\mathcal{I}_{B}\cap \mathcal{I}_{baB}^{1}$ is a topological disk, since $\langle X_0,X_0\rangle=\frac{9}{4}-\frac{3\sqrt{3}}{2}<0$.

Writing out $z_j=x_j+iy_j$ for real $x_j$ and $y_j$, the intersection of $\partial_{\infty}(\mathcal{I}_{B}\cap \mathcal{I}_{baB}^{1})\cap \hat{\mathcal{I}}_{b} $ is described by the solutions of the following system

$$\left\{
\begin{aligned}
&\frac{11}{2}-2x_1+4y_1-2x_2-\frac{x_1x_2}{2}-y_1x_2+x_1y_2-\frac{y_1y_2}{2}=0 \\
&x_1-2y_1+x_2-\frac{x_1x_2}{2}+y_1x_2-x_1y_2-\frac{y_1y_2}{2}-\frac{3}{2}=0 \\
&x_1^2 + y_1^2-1=0\\
& x_2^2 + y_2^2-1=0
\end{aligned}
\right.
$$

It is easy to check that this system has no real solutions. That is, the intersection of $\partial_{\infty}(\mathcal{I}_{B}\cap \mathcal{I}_{baB}^{1})\cap \hat{\mathcal{I}}_{b} $ is empty. One can exhibit a single point of $\mathcal{I}_{B}\cap \mathcal{I}_{baB}^{1}$ inside the isometric sphere of $b$.
Therefore $\mathcal{I}_{B}\cap \mathcal{I}_{baB}^{1}$ is entirely contained in the interior of the isometric sphere of $b$.
\end{proof}

\begin{prop} $\mathcal{I}_{B}\cap \mathcal{I}_{Bab}^{1} \cap D_{S^*}$  is a combinatorial triangle.
\end{prop}
\begin{proof}
The Giraud disk $\mathcal{I}_{B}\cap \mathcal{I}_{Bab}^{1}$ can be parameterized by negative vector of the form
$$V=(\overline{z}_1p_0-b(p_0))\boxtimes (\overline{z}_2p_0-ABAb(p_0)),$$ where $|z_1|=|z_2|=1$.
Explicitly, this $V$ can be written as $V(z_1,z_2)=v_0+z_1v_1+z_2v_2$, where
\begin{align*}
v_1&=\left(\frac{1}{2},\frac{i}{2},1+\frac{i}{2}\right),\\
v_2&=\left(\frac{-1-i}{2},-\frac{1}{2}-i,0\right), \\
v_3&=\left(\frac{-1+i}{2},\frac{1}{2},0\right).
\end{align*}
Furthermore, $\langle V,V\rangle<0$ can be written as
$${\rm Re} \left(\frac{11}{4}-\frac{5}{2}z_1-\left(\frac{1}{2}-i\right)z_2-\left(\frac{1}{2}-i\right)z_2\overline{z}_1\right)<0.$$
In order to encode the combinatorial structure of this ridge, we need to study the intersection of this Giraud disk with the isometric spheres given in Table \ref{table:intersecton}. The equation of the intersection with $\mathcal{I}_g$ is given by
$$|\langle V(z_1,z_2),p_0\rangle|^2=|\langle V(z_1,z_2),g^{-1}(p_0)\rangle|^2.$$

 We start by studying the intersection of $\hat{\mathcal{I}}_{B}\cap \hat{\mathcal{I}}_{Bab}^{1}$ with $\hat{\mathcal{I}}_{Aba}$.
Writing out $z_j=x_j+iy_j$ for real $x_j$ and $y_j$, the intersection $\partial_\infty(\mathcal{I}_{B}\cap \mathcal{I}_{Bab}^{1})\cap \hat{\mathcal{I}}_{Aba}$
is described by the solutions of the system

$$\left\{
\begin{aligned}
&y_1-x_2y_1-2y_2+y_1y_2+x_1(1+x_2+y_2)=0 \\
&11-2x_2+4x_2y_1-4y_2-2y_1y_2-10x_1-2x_1x_2-4x_1y_2=0 \\
&x_1^2 + y_1^2-1=0\\
& x_2^2 + y_2^2-1=0
\end{aligned}
\right.
$$

This system has exactly two solutions, given approximately by
$$x_1=0.378005..., y_1=-0.925803..., x_2 =0.960944...,
 y_2=0.276744...$$ and $$x_1=0.987712..., y_1=0.156285..., x_2=-0.872037...,
 y_2=0.48944....$$
So the curve $\hat{\mathcal{I}}_{B}\cap \hat{\mathcal{I}}_{Bab}^{1}\cap \hat{\mathcal{I}}_{Aba}$ intersects $\partial_\infty(\mathcal{I}_{B}\cap \mathcal{I}_{Bab}^{1})$ in two points.

Similarly,  the curve $\hat{\mathcal{I}}_{B}\cap \hat{\mathcal{I}}_{Bab}^{1}\cap \hat{\mathcal{I}}_{BAB}$ intersects  $\partial_\infty(\mathcal{I}_{B}\cap \mathcal{I}_{Bab}^{1})$ in two points given approximately by
$$x_1=0.987712..., y_1=-0.156285..., x_2=-0.784829...,
 y_2=0.619713...$$ and $$x_1=0.378005..., y_1=0.925803..., x_2=0.107031...,
 y_2=0.994256....$$

Now the triangle in  part (e) of  Figure \ref{fig:B}  has three boundary arcs, given by $\hat{\mathcal{I}}_{B}\cap \hat{\mathcal{I}}_{Bab}^{1}\cap \hat{\mathcal{I}}_{Aba}$, $\hat{\mathcal{I}}_{B}\cap \hat{\mathcal{I}}_{Bab}^{1}\cap \hat{\mathcal{I}}_{BAB}$ and $\partial_\infty(\mathcal{I}_{B}\cap \mathcal{I}_{Bab}^{1})$.

Next, we include the element $bAB$ in Table 2. From the similar calculation as  above, we find that the curve $\hat{\mathcal{I}}_{B}\cap \hat{\mathcal{I}}_{Bab}^{1}\cap \hat{\mathcal{I}}_{bAB}$ does not intersect the boundary arcs of the triangle. This implies that the boundary of this triangle
is either completely inside or outside the isometric sphere of $bAB$. It is easy to check that it is outside by testing a sample point. We can also prove that the curve $\hat{\mathcal{I}}_{B}\cap \hat{\mathcal{I}}_{Bab}^{1}\cap \hat{\mathcal{I}}_{bAB}$ does not intersect the interior of the  triangle by computing the critical points of the equation for the curve  $\hat{\mathcal{I}}_{B}\cap \hat{\mathcal{I}}_{Bab}^{1}\cap \hat{\mathcal{I}}_{bAB}$.

For the element $aBA$ in Table \ref{table:intersecton}, one can verify that the curve  $\hat{\mathcal{I}}_{B}\cap \hat{\mathcal{I}}_{Bab}^{1}\cap \hat{\mathcal{I}}_{aBA}$ does not intersect $\partial_\infty(\mathcal{I}_{B}\cap \mathcal{I}_{Bab}^{1})$. By taking a sample point, one get that the Giraud disk $\mathcal{I}_{B}\cap \mathcal{I}_{Bab}^{1}$  lies outside  the isometric sphere of $aBA$.

Similar arguments allow us to handle the detailed study of all elements in Table \ref{table:intersecton}.

\end{proof}

\begin{figure}[htbp]
  \centering
  \subfigure[Ridge on $\mathcal{I}_B\cap \mathcal{I}_{Aba}$ ]{\includegraphics[width=0.29\textwidth]{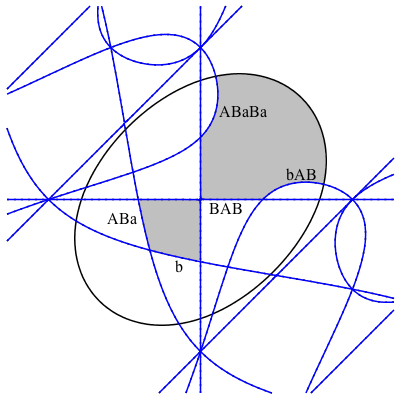}} \hfill
  \subfigure[Ridge on $\mathcal{I}_B\cap \mathcal{I}_{abA}$]{\includegraphics[width=0.3\textwidth]{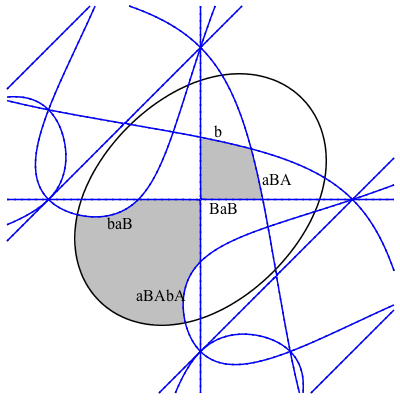}}\hfill
  \subfigure[Ridge on $\mathcal{I}_B\cap \mathcal{I}_{ABa}$]{\includegraphics[width=0.29\textwidth]{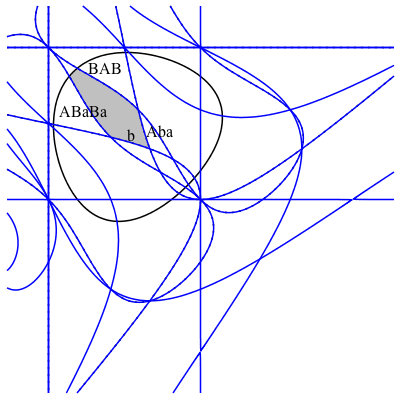}}\\
  \subfigure[Ridge on $\mathcal{I}_B\cap \mathcal{I}_{aBA}$]{\includegraphics[width=0.29\textwidth]{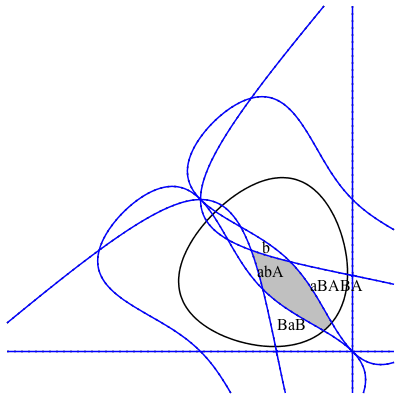}}\hfill
  \subfigure[Ridge on $\mathcal{I}_B\cap \mathcal{I}_{ABaba}$]{\includegraphics[width=0.29\textwidth]{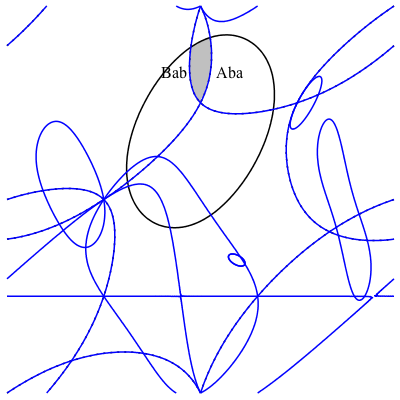}}\hfill
  \subfigure[Ridge on $\mathcal{I}_B\cap \mathcal{I}_{aBAbA}$]{\includegraphics[width=0.29\textwidth]{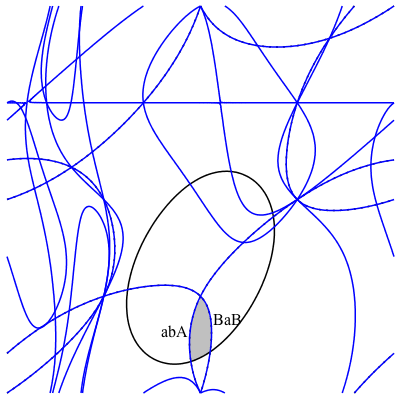}}\\
   \subfigure[Ridge on $\mathcal{I}_B\cap \mathcal{I}_{aBABA}$]{\includegraphics[width=0.29\textwidth]{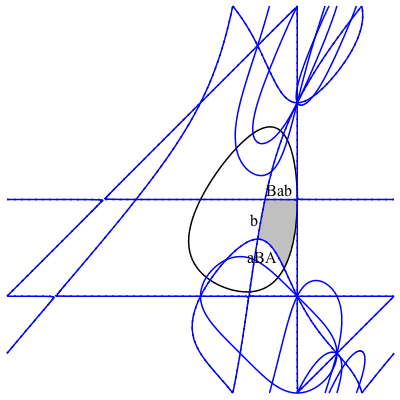}} \hfill
  \subfigure[Ridge on $\mathcal{I}_B\cap \mathcal{I}_{ABaBa}$]{\includegraphics[width=0.29\textwidth]{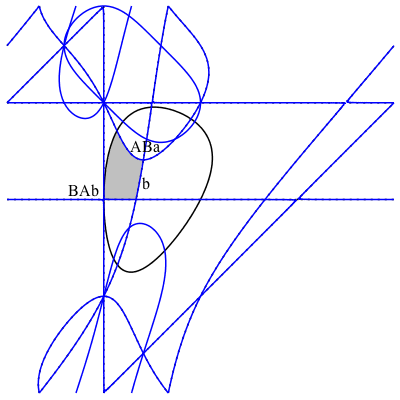}}\hfill
  \subfigure[Ridge on $\mathcal{I}_B\cap \mathcal{I}_{b}$]{\includegraphics[width=0.29\textwidth]{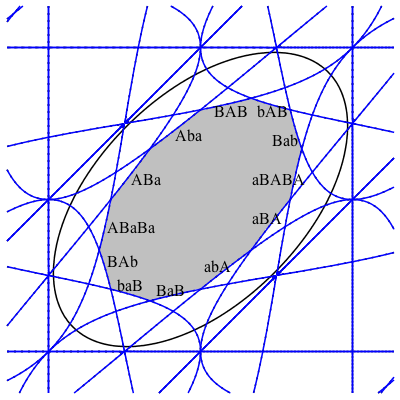}}\\
   \subfigure[Ridge on $\mathcal{I}_B\cap \mathcal{I}_{baB}$]{\includegraphics[width=0.29\textwidth]{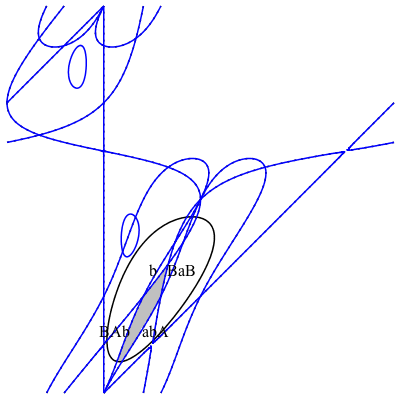}} \hfill
  \subfigure[Ridge on $\mathcal{I}_B\cap \mathcal{I}_{bAB}$]{\includegraphics[width=0.29\textwidth]{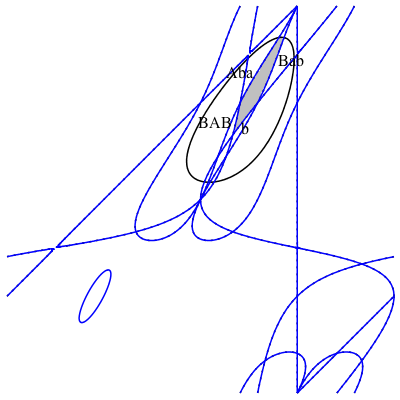}}\hfill
  \subfigure[Ridge on $\mathcal{I}_B\cap \mathcal{I}_{Bab}$]{\includegraphics[width=0.29\textwidth]{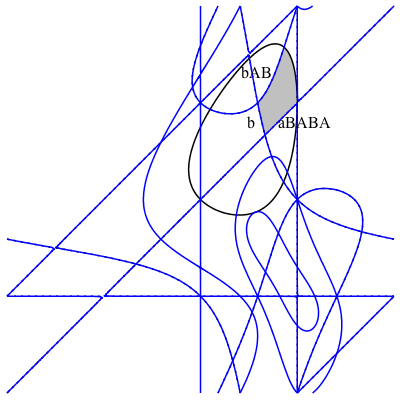}}\\
   \subfigure[Ridge on $\mathcal{I}_B\cap \mathcal{I}_{BAb}$]{\includegraphics[width=0.29\textwidth]{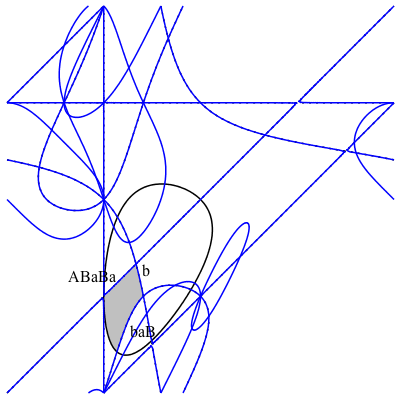}} \hfill
  \subfigure[Ridge on $\mathcal{I}_B\cap \mathcal{I}_{BAB}$]{\includegraphics[width=0.29\textwidth]{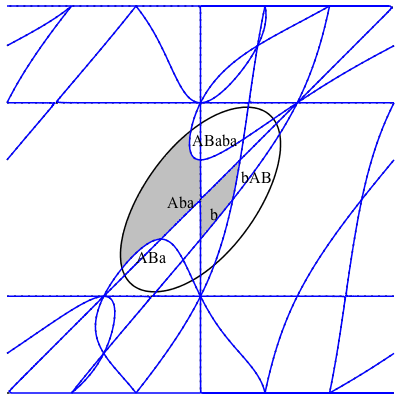}}\hfill
  \subfigure[Ridge on $\mathcal{I}_B\cap \mathcal{I}_{BaB}$]{\includegraphics[width=0.29\textwidth]{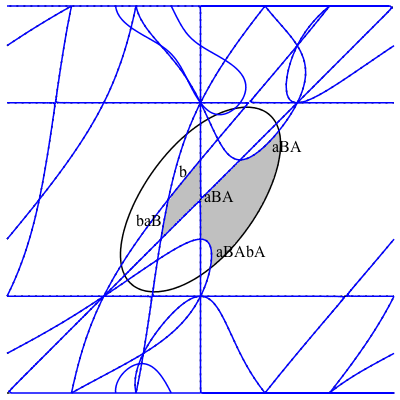}}
  \caption{All the ridges on the side $\mathcal{I}_B$.}\label{fig:B}
\end{figure}

\begin{figure}[htbp]
  \centering
  \subfigure[Ridge on $\mathcal{I}_{bAB}\cap \mathcal{I}_{Aba}$ ]{\includegraphics[width=0.3\textwidth]{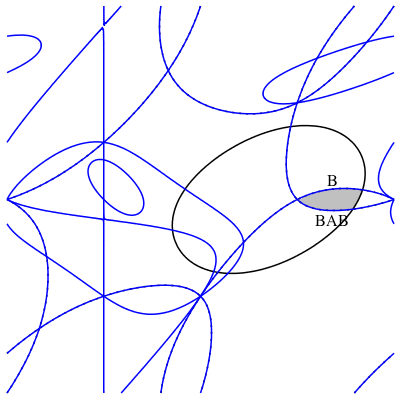}} \hfill
  \subfigure[Ridge on $\mathcal{I}_{bAB}\cap \mathcal{I}_{B}$]{\includegraphics[width=0.3\textwidth]{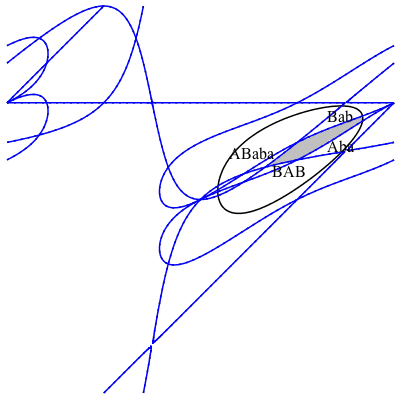}}\hfill
  \subfigure[Ridge on $\mathcal{I}_{bAB}\cap \mathcal{I}_{b}$]{\includegraphics[width=0.3\textwidth]{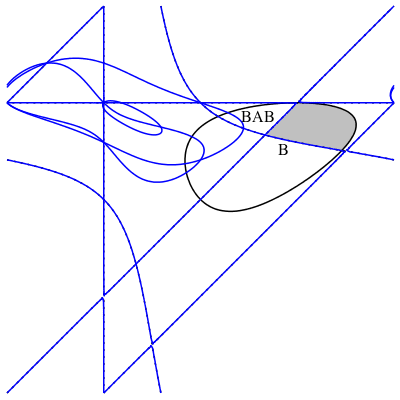}}\\
  \subfigure[Ridge on $\mathcal{I}_{bAB}\cap \mathcal{I}_{BAB}$]{\includegraphics[width=0.3\textwidth]{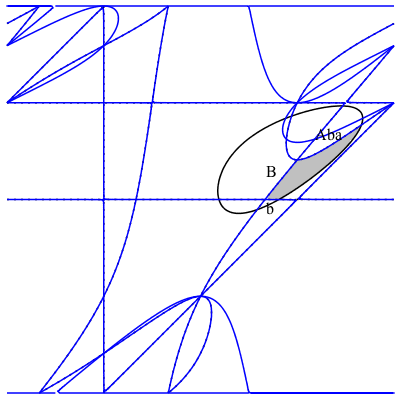}}
  \subfigure[Ridge on $\mathcal{I}_{bAB}\cap \mathcal{I}_{Bab}$]{\includegraphics[width=0.3\textwidth]{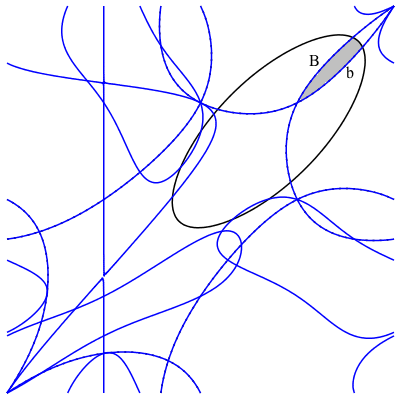}}
  \caption{All the ridges on the side $\mathcal{I}_{bAB}$.}\label{fig:bAB}
\end{figure}

\begin{figure}[htbp]
  \centering
  \subfigure[Ridge on $\mathcal{I}_{bab}\cap \mathcal{I}_{Aba}$ ]{\includegraphics[width=0.3\textwidth]{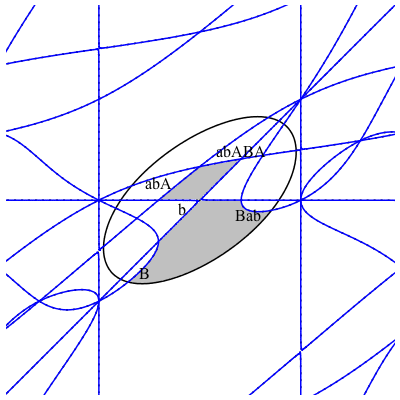}} \hfill
  \subfigure[Ridge on $\mathcal{I}_{bab}\cap \mathcal{I}_{abA}$]{\includegraphics[width=0.3\textwidth]{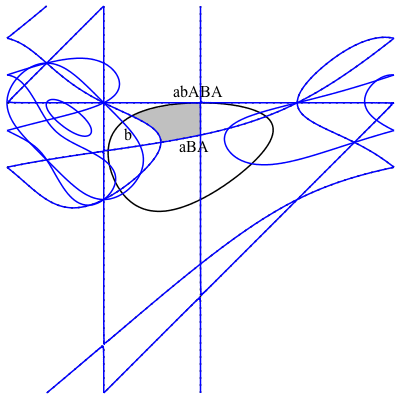}}\hfill
  \subfigure[Ridge on $\mathcal{I}_{bab}\cap \mathcal{I}_{ABa}$]{\includegraphics[width=0.3\textwidth]{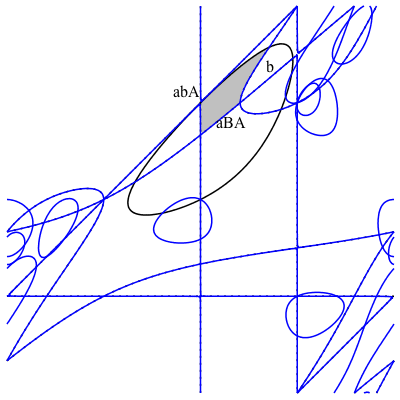}}\\
  \subfigure[Ridge on $\mathcal{I}_{bab}\cap \mathcal{I}_{aBA}$]{\includegraphics[width=0.3\textwidth]{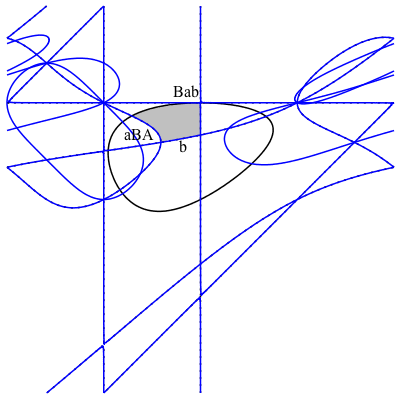}}\hfill
  \subfigure[Ridge on $\mathcal{I}_{bab}\cap \mathcal{I}_{ABaba}$]{\includegraphics[width=0.3\textwidth]{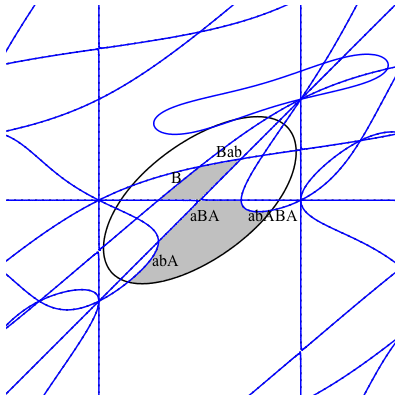}}\hfill
  \subfigure[Ridge on $\mathcal{I}_{bab}\cap \mathcal{I}_{aBAbA}$]{\includegraphics[width=0.3\textwidth]{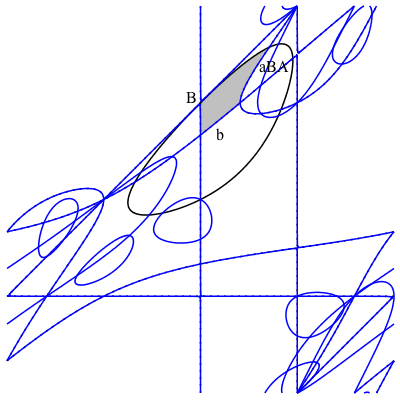}}\\
  \caption{All the ridges on the side $\mathcal{I}_{bab}$.}\label{fig:bab}
\end{figure}

\begin{figure}[htbp]
  \centering
  \subfigure[Ridge on $\mathcal{I}_{bAb}\cap \mathcal{I}_{ABa}$ ]{\includegraphics[width=0.3\textwidth]{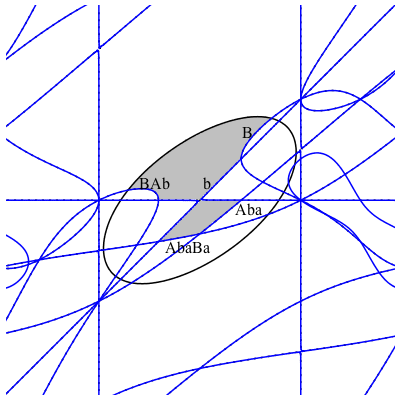}} \hfill
  \subfigure[Ridge on $\mathcal{I}_{bAb}\cap \mathcal{I}_{Aba}$]{\includegraphics[width=0.3\textwidth]{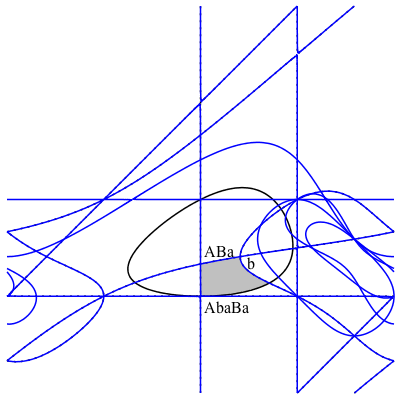}}\hfill
  \subfigure[Ridge on $\mathcal{I}_{bAb}\cap \mathcal{I}_{AbaBa}$]{\includegraphics[width=0.3\textwidth]{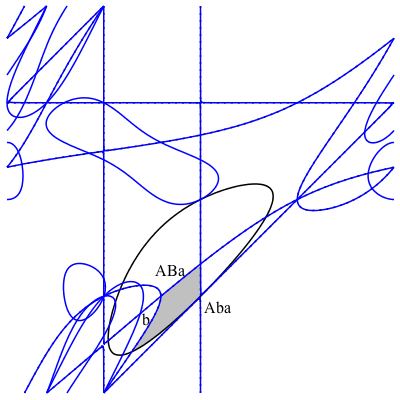}}\\
  \subfigure[Ridge on $\mathcal{I}_{bAb}\cap \mathcal{I}_{B}$]{\includegraphics[width=0.3\textwidth]{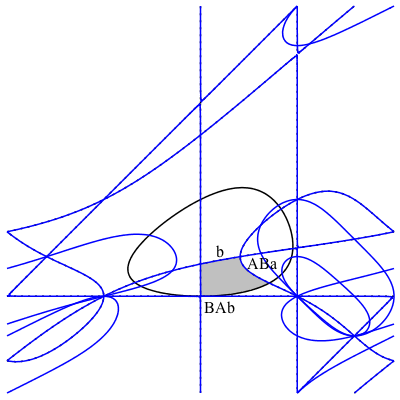}}\hfill
  \subfigure[Ridge on $\mathcal{I}_{bAb}\cap \mathcal{I}_{b}$]{\includegraphics[width=0.3\textwidth]{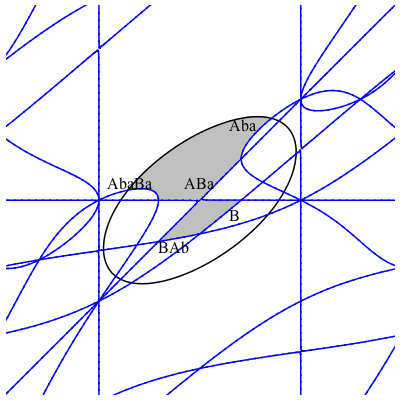}}\hfill
  \subfigure[Ridge on $\mathcal{I}_{bAb}\cap \mathcal{I}_{BAb}$]{\includegraphics[width=0.3\textwidth]{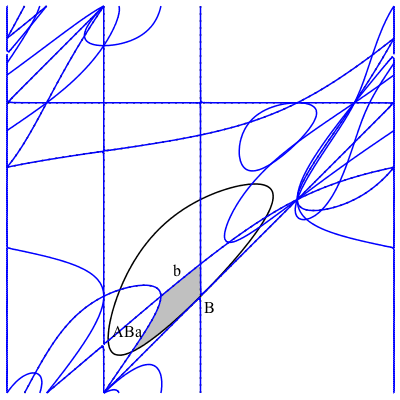}}\\
  \caption{All the ridges on the side $\mathcal{I}_{bAb}$.}\label{fig:bAb}
\end{figure}

\begin{figure}[htbp]
  \centering
  \subfigure[Ridge on $\mathcal{I}_{Bab}\cap \mathcal{I}_{aBA}$ ]{\includegraphics[width=0.3\textwidth]{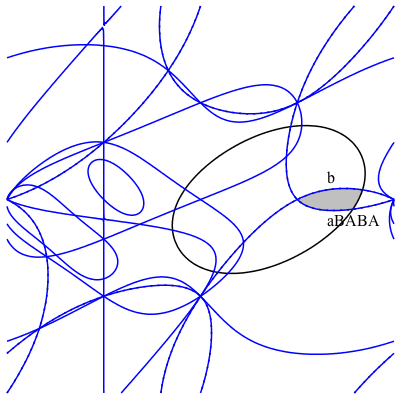}} \hfill
  \subfigure[Ridge on $\mathcal{I}_{Bab}\cap \mathcal{I}_{aBABA}$]{\includegraphics[width=0.3\textwidth]{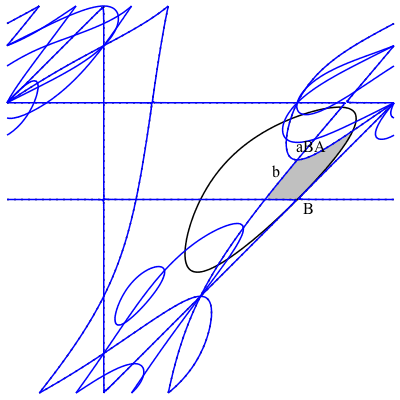}}\hfill
  \subfigure[Ridge on $\mathcal{I}_{Bab}\cap \mathcal{I}_{B}$]{\includegraphics[width=0.3\textwidth]{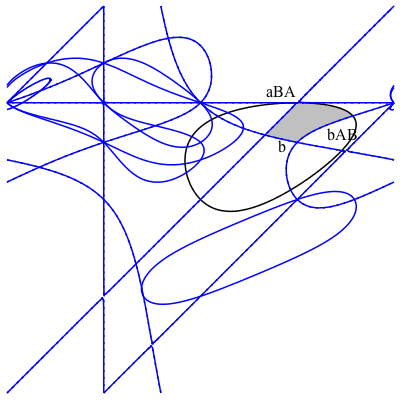}}\\
  \subfigure[Ridge on $\mathcal{I}_{Bab}\cap \mathcal{I}_{b}$]{\includegraphics[width=0.3\textwidth]{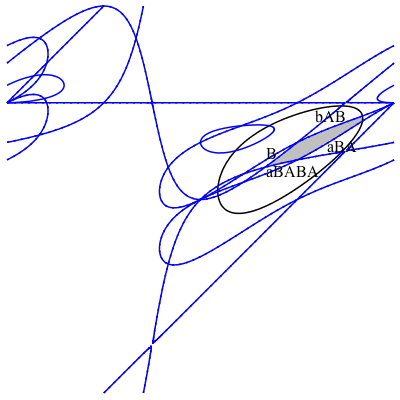}}
  \subfigure[Ridge on $\mathcal{I}_{Bab}\cap \mathcal{I}_{bAB}$]{\includegraphics[width=0.3\textwidth]{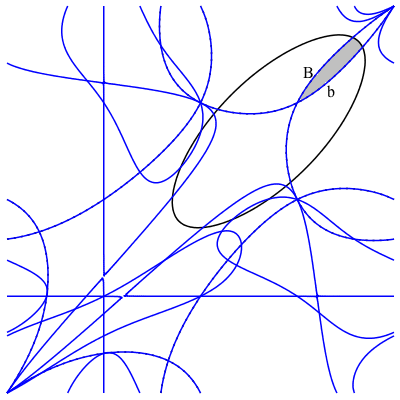}}
  \caption{All the ridges on the side $\mathcal{I}_{Bab}$.}\label{fig:Bab}
\end{figure}

From Proposition \ref{prop: combinatorics},  the isometric sphere $\mathcal{I}_{B}$ contributes a side of the Ford domain, say $\mathcal{I}_{B}\cap D_{S^*}$, where there are fourteen infinite ridges in $\mathcal{I}_{B}\cap D_{S^*}$.   See Table \ref{table:adjacent344} in Section \ref{section: ford344}, the spinal sphere of $B$ is adjacent to $14=2+5+3+4$ spinal spheres. Comparing also to  Figures 3, 5 and 6   of  \cite{Deraux:2016}.

We write $E$ for $\partial_{\infty}D_{S^*}$, and $C$  for $\partial E$. It is easy to see that  $D_{S^*}, E$ and $C$ are all $A$-invariant by the construction. The map $A$  acts
on the Heisenberg group as a Heisenberg translation preserving $x$-axis. In Heisenberg coordinates, the action of $A$ is given by
$$A(z,t)=(z-2,t+4{\rm Re} z).$$ It follows from  Section \ref{section: ford344}  that $C$ is tiled by some polygons and their orbits under the action of $A$.
The identifications in $C$ coming from  the action of $A$.  We denote by $\sim$ the corresponding equivalence relation on $C$; it is easy to check
 that $C/ \sim$ is a torus. We need to claim that this torus is unknotted. So we  prove that

\begin{prop} The $x$-axis of $\partial {\bf H}^2_{\mathbb C}-\{q_{\infty}\}=\mathbb{C} \times \mathbb{R}$ is contained in the complement of $\partial_{\infty}D_{S^*}$.
\end{prop}
\begin{proof} The fundamental domain of $A$  acting on $x$-axis is a segment parameterized by
$$\iota_x=\{[x,0,0] \in \partial {\bf H}^2_{\mathbb C} ~~| ~~\ x\in [-1,1]\}.
$$
We decompose this segment into two parts as
\begin{align*}
\iota_x^1&=\{[x,0,0]  \in \partial {\bf H}^2_{\mathbb C} ~~| ~~\ x\in [-1,0]\},  \\
\iota_x^2&=\{[x,0,0]  \in \partial {\bf H}^2_{\mathbb C} ~~| ~~\ x\in [0,1]\}.
\end{align*}

Note that a spinal sphere is convex. It is easy to check that the interval $\iota_x^1$ is in the interior of the isometric sphere $\mathcal{S}_{Aba}$ and the interval $\iota_x^2$ is in the interior of the isometric sphere  $\mathcal{S}_{B}$.
\end{proof}

\section{Manifold at infinity of the complex hyperbolic triangle group  $\Delta_{3,4,4;\infty}$}\label{section: ford344}

In this section, we study the manifold  at infinity of the even subgroup of the complex hyperbolic triangle group  $\Delta_{3,4,4;\infty}$.

Let $I_{1}$, $I_{2}$ and $I_{3}$ be the matrices in section \ref{sec-gens}, we have $I_{1}I_{2}$ is parabolic, $(I_{2}I_{3})^3=(I_{3}I_{1})^4=id$, and $(I_{1}I_{3}I_{2}I_{3})^4=(I_{1}I_{2}I_{3}I_{2})^4=id$.
Let $A=I_{1}I_{2}$ and  $B=I_{2}I_{3}$, then $B^3=id$, $(AB)^4=id$ and  $(AB^2)^4=id$,  and remember that we write $a,b$ for $A^{-1}, B^{-1}$ respectively.

\subsection{The Ford domain of the complex hyperbolic triangle group  $\Delta_{3,4,4;\infty}$}

In this section we will show that $D_{S^*}$ is  the Ford domain by using the Poincar\'e polyhedron theorem and study the ideal boundary  of $D_{S^*}$. We will give the detailed information on spinal spheres and infinite ridges, which is enough to get the manifold  at infinity.

Recall $S^*$ in Section \ref{section:comb344} is the set of  the  essential words
$$\{b, B, Bab, BAb, BaB, bAb, BAB, bab, bAB, baB\},$$ which define side-pairings of $D_{S^*}$.

The main task is to check  the Poincar\'e ridge cycles. Recall that a ridge is by definition a codimension-2 facet of $D_{S^*}$.
Since all the complex
spines of our isometric spheres  intersect at $q_{\infty}$, we can think of the  intersections as being coequidistant, and  in particular, their intersections are all smooth disks, equidistant
from three points.  Note that no ridge of $D_{S^*}$ is totally geodesic. The ridges of $D_{S^*}$ are the so-called Giraud disk which are generic intersections of two bisectors.

Because of Giraud's theorem, the ridges of $D_{S^*}$ lie on precisely three bisectors, so the local tiling condition near generic ridges is actually a consequence of the existence of
side-pairing. The ridge cycles can be obtained by computing orbits of these triples of points under the side-pairings; whenever a ridge in the cycle differs from the starting
ridge by a power of $A$, we close the cycle up by that power of $A$.

The bounded ridge on $\mathcal{I}(b)\cap\mathcal{I}(B)$ is sent to itself by $B$. One checks that
$$q_{\infty}\stackrel{B}{\longrightarrow}B(q_{\infty})\stackrel{B}{\longrightarrow}B^2(q_{\infty})=b(q_{\infty})\stackrel{B}{\longrightarrow}q_{\infty}.$$
This clearly gives a cycle transformation of order 3 preserving that ridge, so we get the relation $B^3=id$.

In Tables  \ref{table:ridge344} and \ref{table:circle344}, we give the information about infinite ridge cycles and infinite ridge relations of the group $\Delta_{3,4,4;\infty}$. In Table \ref{table:ridge344}, for example, $(r_1, bAB\cap Bab)$ means that the isometric sphere $\mathcal{I}(bAB)$ of
$bAB$ and  the isometric sphere $\mathcal{I}(Bab)$ of  $Bab$ intersect in a  ridge $r_1$.

From Tables  \ref{table:ridge344}, \ref{table:circle344}, and the information about bounded ridges, we can use the Poincar\'e polyhedron theorem to see that
$D_{S^*}$ is in fact the Ford domain. Furthermore, we have the following proposition.

\begin{prop}
No  elliptic element of $\Gamma_1$ fixes any point in $\partial {\bf H}^2_{\mathbb C}$ and the only parabolic elements in $\Gamma_1$ are conjugates of powers of $A$.
\end{prop}
\begin{proof}
From the Poincar\'e ridge cycles, we know that any elliptic element in the group $\Gamma_1$ must be conjugate to a power of $B$, a power of $ab$ and a power of $aB$.  $B$ is a regular elliptic of order 3, thus it does not fix any point in $\partial {\bf H}^2_{\mathbb C}$. As for $ab$ and  $aB$, one can see that the nontrivial, nonregular elliptic elements are $(ab)^2$ and  $(aB)^2$. But these are reflections in points, which have no fixed points in $\partial {\bf H}^2_{\mathbb C}$.

It is easy to see that the ideal vertices all have trivial stabilizers. The only parabolic elements in $\Gamma_1$ are elements in the cyclic subgroup generated by $A$.
\end{proof}
The ideal boundary of $D_{S^*}$ is  $E$, that is $E=\partial_{\infty}D_{S^*}$, and the boundary of the 3-manifold  $E$ is $C$, that  is $C=\partial E$. $C$ is an infinite annulus.

\begin{figure}[htbp]
	\center{\scalebox{0.3}[0.3]{\includegraphics {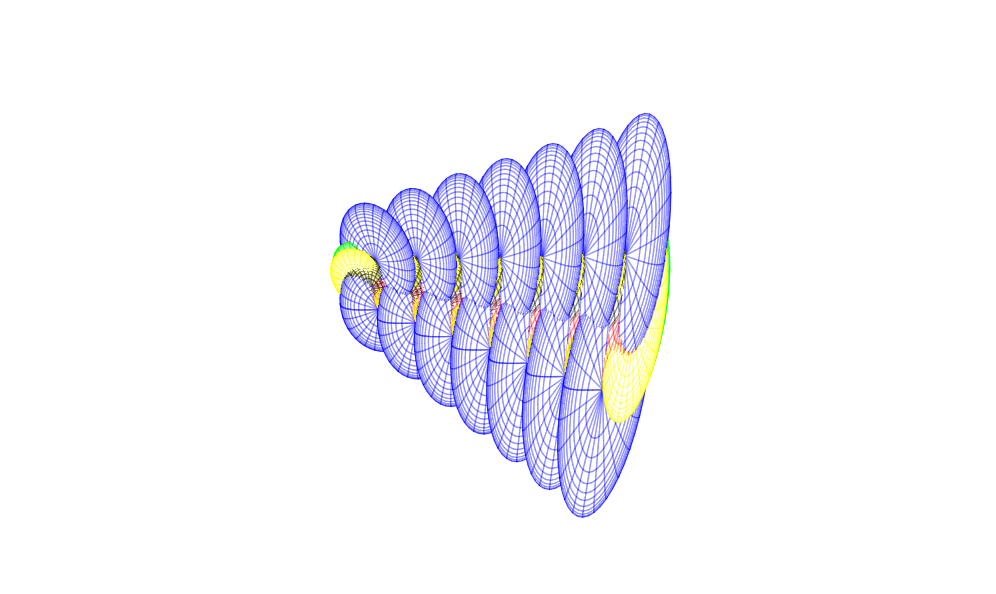}}}
	\caption{Realistic view  of the ideal boundary of the Ford domain of $\Delta_{3,4,4;\infty}$.}
\label{figure:front4}
\end{figure}


We will show $C$ consists of spinal spheres  of  the following elements:

  1. $A$-translations of $\{B, b\}$, which are the  blue  colored polygons in Figure \ref{figure:front4}. Each of the spinal spheres of  $\{B, b\}$ contributes a tetradecagon of the ideal boundary of the Ford domain. See Table \ref{table:adjacent344} for more details;

  2. $A$-translations of $\{bAB, baB\}$, which are the   red   colored polygons in Figure \ref{figure:front4}. Each of the spinal spheres of  $\{bAB, baB\}$ contributes two triangles  of the  ideal  boundary of the Ford domain;

   3. $A$-translations of $\{bAb, BaB\}$, which are the  green colored polygons in Figure \ref{figure:front4}.   Each of the spinal spheres of $\{bAb, BaB\}$ contributes a pentagon   of the  ideal  boundary of the Ford domain;

   4. $A$-translations of $\{bab, BAB\}$, which are the  yellow  colored polygons in Figure \ref{figure:front4}. Each of the spinal spheres of $\{bab, BAB\}$ contributes two triangles of the  ideal  boundary of the Ford domain;

   5. $A$-translations of  $\{Bab, BAb\}$, which are the  black colored polygons in Figure \ref{figure:front4}.  Each of the spinal spheres of $\{Bab, BAb\}$ contributes a pentagon  of the   ideal  boundary of the Ford domain.

Figure \ref{figure:front4}  is a realistic  view  of the ideal boundary of the Ford domain of $\Delta_{3,4,4;\infty}$. This figure is just for motivation, and we need to show the figure is correct.

In Table \ref{table:adjacent344}, we give the detailed information of the contributions of the spinal spheres for the ideal boundary of the Ford domain.  $(bAb)_1$ and $(bAb)_2$ means the spinal sphere of $bAb$ contributes two parts of the boundary of   $D_{S^*} \cap \partial {\bf H}^2_{\mathbb C}$.  We see that  the ideal boundary  of the Ford domain of the group $\Delta_{3,4,4;\infty}$ is the complement of a $D^2 \times (-\infty, \infty)$ in $\mathbb{C} \times \mathbb{R} = \partial {\bf H}^2_{\mathbb C}- \{q_{\infty}\}$.

From Table \ref{table:adjacent344}, we get Figure \ref{figure:boundary4f}, which gives us an abstract picture of the boundary of the ideal boundary of Ford domain of the complex hyperbolic triangle group  $\Delta_{3,4,4;\infty}$. That is,  Figure \ref{figure:boundary4f} gives us an abstract picture of $C$. Since $(AB)^4=id$, the isometric sphere $\mathcal{I}(bab)$ of  $bab$ is the same as the isometric sphere of  $ABABa$. So in Figure \ref{figure:boundary4f}, we only see one part which is labeled by $bab$, even though we have claimed that the isometric sphere of $bab$ contribute two triangles of the boundary of the ideal boundary of  the Ford domain.
Note that $A$ acts on the picture horizontally to the left, and two zigzag lines with angles $\frac{3\pi}{4}$ to the positive horizontal axis are glued together to get the  infinite annulus $C$.

\begin{table}[htbp]\label{table: vertices}
		\begin{tabular}{|c|c|}
			\hline
			Spinal sphere & Adjacent spinal spheres in clockwise order  \\
			\hline
$(bAb)_1$ & $b$, $Aba$, $AbaBa$  \\
			\hline
$(bAb)_2$ & $B$, $BAb$, $ABa$  \\
			\hline
$(BaB)_1$ & $B$, $aBA$, $aBAbA$  \\
			\hline
$(BaB)_2$ & $b$, $baB$, $abA$  \\
			\hline
$(bab)_1$ & $B$, $Bab$, $aBA$  \\
			\hline
$(bab)_2$ & $b$, $abA$, $abABA$  \\
			\hline
$(BAB)_1$ & $b$, $bAB$, $Aba$  \\
			\hline
$(BAB)_2$ & $B$, $ABa$, $ABaba$  \\
			\hline
$Bab$ & $b$, $aBA$, $bab$, $B$, $bAB$  \\
			\hline

$BAb$ & $b$, $ABa$, $bAb$, $B$, $baB$ \\
			\hline
 $bAB$ & $B$, $Aba$, $Ababa$, $b$, $Bab$ \\
			\hline
			$(baB)$ & $B$, $abA$, $BaB$, $b$, $BAb$ \\
			\hline
$B$ & $Aba$, $bAB$, $Bab$, $bab$, $aBA$, $abAbA$, $aBAbA$, \\
& $abA$, $baB$, $BAb$,   $bAb$, $ABa$, $Ababa$, $ABaba$ \\
			\hline
$b$ & $bab$, $abABA$, $aBA$, $Bab$, $bAB$, $BAB$, $Aba$,   \\
 & $bAb$, $AbaBa$, $ABa$, $BAb$, $baB$, $abAbA$, $abA$  \\
			\hline
		\end{tabular}
\caption{Adjacent relations in the ideal boundary of the Ford domain boundary of  $\Delta_{3,4,4;\infty}$.}
\label{table:adjacent344}
\end{table}



\begin{figure}[htbp]
	\scalebox{0.25}[0.25]{\includegraphics {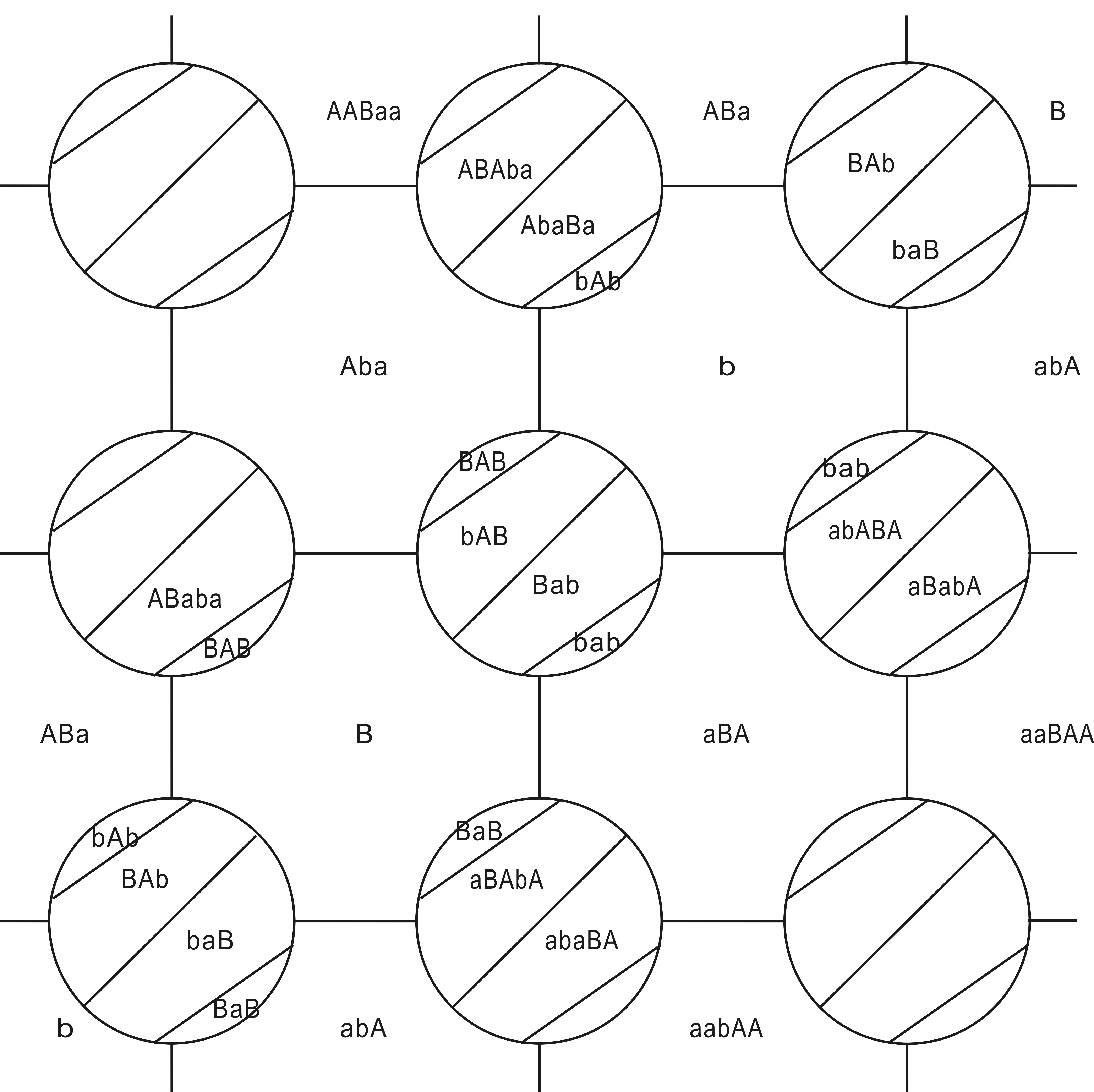}}
	\caption{A combinatorial picture of the boundary of the ideal boundary of the Ford domain of the complex hyperbolic triangle group  $\Delta_{3,4,4;\infty}$.}
\label{figure:boundary4f}
\end{figure}

\subsection{A 2-spine of the manifold at infinity of the  complex hyperbolic triangle group $\Delta_{3,4,4;\infty}$}

In this subsection, we give a  canonical 2-spine $S$ of the manifold at infinity of the  complex hyperbolic triangle group  $\Delta_{3,4,4;\infty}$,  from which we can compute
the fundamental group of $M$, our manifold at infinity of the complex hyperbolic triangle group  $\Delta_{3,4,4;\infty}$.


The ideal  boundary $D_{S^*} \cap (\partial{\mathbf{H}^2_{\mathbb C}}-\{q_{\infty}\})$  of the Ford domain of $\Delta_{3,4,4;\infty}$ is a solid cylinder homeomorphic to $(\mathbb{C}- int (\mathbb{D})) \times \mathbb{R}^{1}$, where $\mathbb{D}$ is the unit disk in $\mathbb{C}$. Where Figure \ref{figure:boundary4f} is an abstract picture of $\partial  \mathbb{D} \times \mathbb{R}^{1}$, that is,  an abstract picture of $C=\partial E$.

Since  we have exact one parabolic fixed point, which is  $q_{\infty}$. In other words, there is no parabolic fixed point in
$\partial  \mathbb{D} \times \mathbb{R}^{1}$.  The 3-manifold $M$ is a quotient space of   $D_{S^*} \cap (\partial{\mathbf{H}^2_{\mathbb C}}-\{q_{\infty}\})$,  where we first consider the equivalence  given by the $A$-action,  and then the equivalence  given by side-pairings of the spinal spheres.

Each infinite ridge in the boundary of the Ford domain gives us an edge in the boundary 3-manifold. We take a fundamental domain of $A$ acts on the infinite annulus, which in Figure \ref{figure:boundary4forientation} is the largest region  bounded by the edges with labels. It consists of the following parts in Figure \ref{figure:boundary4f}:

(1). Two big tetradecagons correspond to the spinal spheres of the elements $\{B, b\}$;

(2). Four triangles correspond to the spinal spheres of the elements $\{BAB, bab\}$, $\{bAb, BaB\}$;

(3). Four  pentagons correspond to the spinal spheres of the elements $\{bAB, baB\}$, $\{BAb, Bab\}$.

 Note that we label the edges from 1 to 45, but in fact the edge $e_{35}$ and the edge $e_{22}$ are glued together in the infinite annulus $C$ of the boundary of ideal boundary of  the Ford domain. The same as edges $e_{36}$ and $e_{21}$, edges $e_{37}$ and $e_{20}$. In total, there are 42 edges.

Note also  that the orientations of edges in Figure \ref{figure:boundary4forientation} can be obtained from the ridge relations with taking care of the neighbouring isometric spheres.  For example in the procedure of  $(e_{44}, aBAbA\cap B) \rightarrow (e_{33}, b \cap AbaBa) \rightarrow (e_{18}, abA\cap baB)$, the positive end of $e_{44}$ is adjacent to  the spinal sphere of $abA$,  and when we use the map $B$, we get $abAB^{-1}=abAb$, but $abAb$ and $bAb$ has the same spinal sphere, so the  positive end of $e_{33}$  is adjacent to  the spinal sphere $bAb$.

\begin{figure}[htbp]
	\scalebox{0.20}[0.20]{\includegraphics {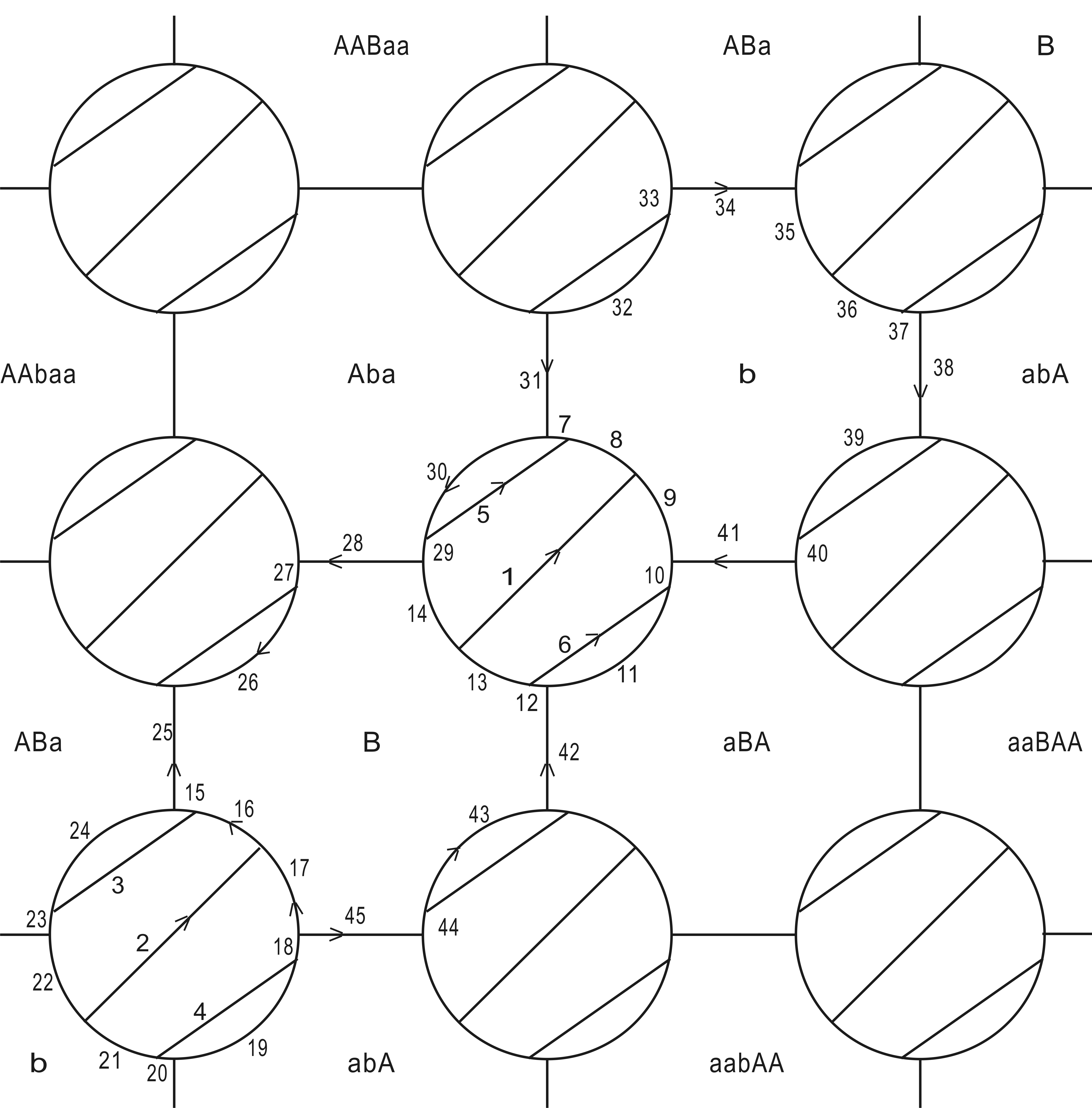}}
	\caption{Fundamental domain and edge cycles in  the boundary of the ideal boundary of the Ford domain of the complex hyperbolic triangle group $\Delta_{3,4,4;\infty}$.}
\label{figure:boundary4forientation}
\end{figure}

\begin{table}[htbp]\label{table: vertices}
	\begin{tabular}{|c|c|}
		\hline
		Number& Ridge relation   \\
		\hline
		$R_{1}$ & $(r_1, bAB\cap Bab)\rightarrow (r_{23}, BAb\cap ABa)\rightarrow (r_{44}, aBAbA\cap B)$ \\
		& $\rightarrow (r_{33}, b \cap AbaBa) \rightarrow (e_{18}, abA\cap baB) \rightarrow (r_1, bAB\cap Bab)$
		\\
		\hline
		$R_{2}$ & $ (r_2, BAb\cap baB)\rightarrow (r_{29}, bAB\cap Aba)\rightarrow (r_{40}, abABA\cap b)$ \\
		&    $\rightarrow (r_{27}, B\cap ABaba)\rightarrow (r_{10},  aBA\cap Bab)  \rightarrow (r_2, BAb\cap baB)$
		\\
		\hline
		$R_{3}$ & $(r_3, bAb\cap BAb)\rightarrow (r_{13}, Bab\cap B)\rightarrow (r_{37}, b\cap abAbA)$\\
		&$ \rightarrow (r_3, bAb\cap BAb)$
		\\
		\hline
		$R_{4}$ & $(r_4, baB\cap BaB)\rightarrow (r_{15}, bAb\cap B)\rightarrow (r_{8}, b\cap bAB)$\\
		&$ \rightarrow (r_4, baB\cap BaB)$
		 \\
		\hline
		$R_{5}$ & $(r_5, BAB\cap bAB)\rightarrow (r_{36}, baB\cap b)\rightarrow (r_{12},  B\cap bab) $\\
		&$\rightarrow (r_5, BAB\cap bAB)$
		\\
		\hline
		$R_{6}$ & $(r_6, Bab\cap bab)\rightarrow (r_{7}, BAB\cap b)\rightarrow (r_{16}, B\cap BAb)$\\
		&$ \rightarrow (r_6, Bab\cap bab)$
		\\
		\hline
		$R_{7}$ & $(r_9, Bab\cap b)\rightarrow (r_{25}, B\cap ABa)\rightarrow (r_{42}, aBA\cap B) $\\
		&$\rightarrow(r_{35}, b \cap BAb) \rightarrow (r_{22}, b\cap BAb) \rightarrow (r_9, Bab\cap b)$
		\\
		\hline
		$R_{8}$ & $(r_{11}, bab\cap aBA)\rightarrow (r_{26}, Ababa\cap B)\rightarrow (r_{41}, b\cap aBA)$\\
		&$ \rightarrow(r_{28}, Aba \cap B) \rightarrow (r_{39}, b\cap bab) \rightarrow (r_{30}, Aba\cap BAB) $\\
		&$\rightarrow (r_{11}, bab\cap aBA)$
		\\
		\hline
		$R_{9}$ & $(r_{14}, bAB\cap B)\rightarrow (r_{38}, b\cap abA)\rightarrow$ \\
		& $(r_{31}, Aba\cap b) \rightarrow (r_{17}, B \cap baB) \rightarrow (r_{14}, bAB\cap B)$
		\\
		\hline
		$R_{10}$ & $(r_{19}, abA\cap BaB)\rightarrow (r_{24}, bAb\cap ABa)\rightarrow (r_{43}, BaB\cap B) $\\
		& $\rightarrow(r_{34}, b\cap ABa) \rightarrow (r_{45}, abA\cap B) \rightarrow (r_{32}; b\cap bAb)$\\
		&$ \rightarrow (r_{19}, abA\cap BaB)$
		\\
		\hline
	\end{tabular}
	\caption{Infinite ridges of the Ford domain of $\Delta_{3,4,4;\infty}$.}
	\label{table:ridge344}
\end{table}

\begin{table}[htbp]\label{table: vertices}
	\begin{tabular}{|c|c|c|}
		\hline
		Number & Ridge cycle & Cycle relation \\
		\hline
		$R_{1}$ &  $(baB)*a*B*a*(Bab)$  &$(aB)^{4}=id$\\
		\hline
		$R_{2}$ &   $(Bab)*a*b*a*(baB)$ &$ (ab)^{4}=id$\\
		\hline
		$R_{3}$ & $(BaB)*B*(BAb)$  &$ B^{3}=id$
		\\
		\hline
		$R_{4}$ &  $(bAB)*B*(BaB)$ &$B^{3}=id$\\
		\hline
		$R_{5}$  & $(bab)*b*(bAB)$ & $ b^{3}=id$
		\\
		\hline
		$R_{6}$& $(BAb)*b*(bab)$&  $ b^{3}=id$
		\\
		\hline
		$R_{7}$  & $(BAb)*B*a*b$ & $id$\\
		\hline
		$R_{8}$  & $(BAB)*A*B*A*B*A$ &$ (AB)^{4}=id$\\
		\hline
		$R_{9}$ & $(baB)*b*A*B$  &$id$\\
		\hline
		$R_{10}$ & $a*B*a*B*a*(BaB)$  &$(aB)^{4}=id$\\
		\hline
	\end{tabular}
	\caption{Infinite ridge relations  of the Ford domain of $\Delta_{3,4,4;\infty}$.}
	\label{table:circle344}
\end{table}

The manifold at infinity of $\Gamma_{1}$ is  the quotient space of $E=D_{S^*} \cap (\partial{\mathbf{H}^2_{\mathbb C}}-\{q_{\infty}\})$ by the equivalence  given by $A$-action and side-pairings of the spinal spheres. That is, if the 3-manifold at infinity is $M$, then $M=E/\sim$, where  $\sim$ is the equivalence  given by $A$-action and side-pairings of the spinal spheres.
 Now $\partial  \mathbb{D} \times \mathbb{R}^{1}$ is a deformation retraction of $(\mathbb{C}- int (\mathbb{D})) \times \mathbb{R}^{1}$. That is,  $C= \partial E$ is a deformation retraction of  $E$.  The deformation is $A$-equivalent, we also have the quotient space $S$ of  $C/\sim$, here $\sim$ is the equivalence  given by $A$-action and side-pairings of the spinal spheres. So we get a  2-spine $S$  of our 3-manifold at infinity of $\Delta_{3,4,4;\infty}$.
From the side-pairing and $A$-translation, we get the equivalence classes  of the oriented edges in Figure \ref{figure:boundary4forientation}, then we get Table \ref{table:edgerelation4}, where Tables \ref{table:ridge344} and \ref{table:circle344} are involved.
Note that $S$ is canonical for the  boundary 3-manifold of $\Delta_{3,4,4;\infty}$ from the view point of Ford domain.


\begin{table}[htbp]\label{table: 4edge}
		\begin{tabular}{|c|c|}
			\hline
			Edge & Equivalence class \\
			\hline
$E_1$ & $e_{1}$, $e_{23}$, $e_{44}$, $e_{33}$, $e_{18}$\\
			\hline
$E_2$ & $e_{2}$, $e_{29}$, $e_{40}$, $e_{27}$, $e_{10}$\\
			\hline
$E_3$ & $e_{3}$, $e_{13}$, $e_{37}$, $e_{20}$\\
			\hline
$E_4$ & $e_{4}$, $e_{8}$, $e_{15}$\\
			\hline
$E_5$ & $e_{5}$, $e_{36}$, $e_{21}$, $e_{12}$\\
			\hline
$E_6$ & $e_{6}$, $e_{7}$, $e_{16}$\\
			\hline
$E_7$ & $e_{9}$, $e_{42}$, $e_{35}$, $e_{25}$, $e_{22}$\\
\hline
$E_8$ & $e_{11}$, $e_{26}$, $e_{41}$, $e_{28}$, $e_{39}$, $e_{30}$\\
			\hline
$E_9$ & $e_{14}$, $e_{38}$, $e_{31}$, $e_{17}$\\
			\hline
$E_{10}$ & $e_{19}$, $e_{24}$, $e_{34}$, $e_{43}$, $e_{32}$, $e_{45}$\\
			\hline
		\end{tabular}
\caption{Equivalence classes of edges  in the boundary of the ideal boundary of the Ford domain of $\Delta_{3,4,4;\infty}$.}
\label{table:edgerelation4}
\end{table}

Moveover, in Table \ref{table:vertex4}, we give the   equivalence relations of vertices, we orient each edge $E_{i}$, where $E_{i}^{-}$ and $E_{i}^{+}$  are the negative vertex and positive vertex of $E_{i}$ respectively.

\begin{table}[htbp]
		\begin{tabular}{|c|c|}
			\hline
			Vertex & Equivalence class   \\
			\hline
$V_1$ & $E_{1}^{-}$, $E_{3}^{-}$, $E_{9}^{-}$, $E_{10}^{-}$\\
			\hline
$V_2$ & $E_{1}^{+}$, $E_{4}^{+}$, $E_{7}^{-}$, $E_{10}^{+}$\\
			\hline
$V_3$ & $E_{2}^{+}$, $E_{6}^{+}$, $E_{9}^{+}$, $E_{8}^{-}$\\
			\hline
$V_4$ & $E_{8}^{+}$, $E_{7}^{+}$, $E_{2}^{-}$, $E_{5}^{-}$\\
			\hline
$V_5$ & $E_{3}^{+}$, $E_{4}^{-}$, $E_{6}^{-}$, $E_{5}^{+}$\\
			\hline
		\end{tabular}
\caption{Equivalence relations on  the vertices in Figure \ref{figure:boundary4forientation}.}
\label{table:vertex4}
\end{table}

From Tables \ref{table:edgerelation4}  and \ref{table:vertex4}, we get Figure \ref{figure:1-skeleton4f}.  Which is the 1-skeleton of the canonical 2-spine $S$ of our  3-manifold at infinity of  $\Delta_{3,4,4;\infty}$.  Then the 2-spine $S$ is obtained from the graph in Figure \ref{figure:1-skeleton4f} by attaching a set of disks according to Table \ref{table:4disk}. Now Table  \ref{table:4disk}  can be obtained from Figure \ref{figure:boundary4forientation}.

Note that we just give some of the labels  of spinal spheres in Figure \ref{figure:boundary4f}, and we only give some of the orientations of the edges in Figure \ref{figure:boundary4forientation}, both of these information  can obtained from Tables \ref{table:adjacent344}  and \ref{table:edgerelation4}.

From Figure \ref{figure:1-skeleton4f}, we consider the fundamental group of the graph. We oriented the edges of the graph where both edges $E_1$ and $E_{10}$ are from $V_1$ to $V_2$. Edge $E_2$ is from $V_4$ to $V_3$, and edge $E_8$ is from $V_3$ to $V_4$.  Edge $E_3$ is from $V_1$ to $V_5$, $E_4$ is from $V_5$ to $V_2$. Edge $E_5$ is from $V_4$ to $V_5$, and edge $E_6$ is from $V_5$ to $V_3$. Edge $E_7$ is from $V_2$ to $V_4$, and edge $E_9$ is from $V_1$ to $V_3$. We let
 $\phi_1=E_4*E_7*E_5$,
 $\phi_2=E^{-1}_5*E_2*E^{-1}_6$,
$\phi_3=E_6*E^{-1}_9*E_3$,
  $\phi_4=E^{-1}_3*E_1*E^{-1}_4$,
$\phi_5=E_6*E_8*E_5$,
$\phi_6=E_4*E^{-1}_{10}*E_3$.
Here $E^{-1}_i$ is the inverse path of $E_i$.
And it is easy to see $\phi_{i}$, $1 \leq i \leq 6$, is a  generator set of the fundamental group  of the graph based at the vertex $V_{5}$.

The boundary of the disk corresponds to the side-pairing $baB$ is  $E_{4}*E^{-1}_{10}*E_{3}$, which can be seen from Figure \ref{figure:boundary4forientation} and Table \ref{table:edgerelation4}, for more details, see Table \ref{table:4disk}.

\begin{figure}[htbp]
	\scalebox{0.3}[0.3]{\includegraphics {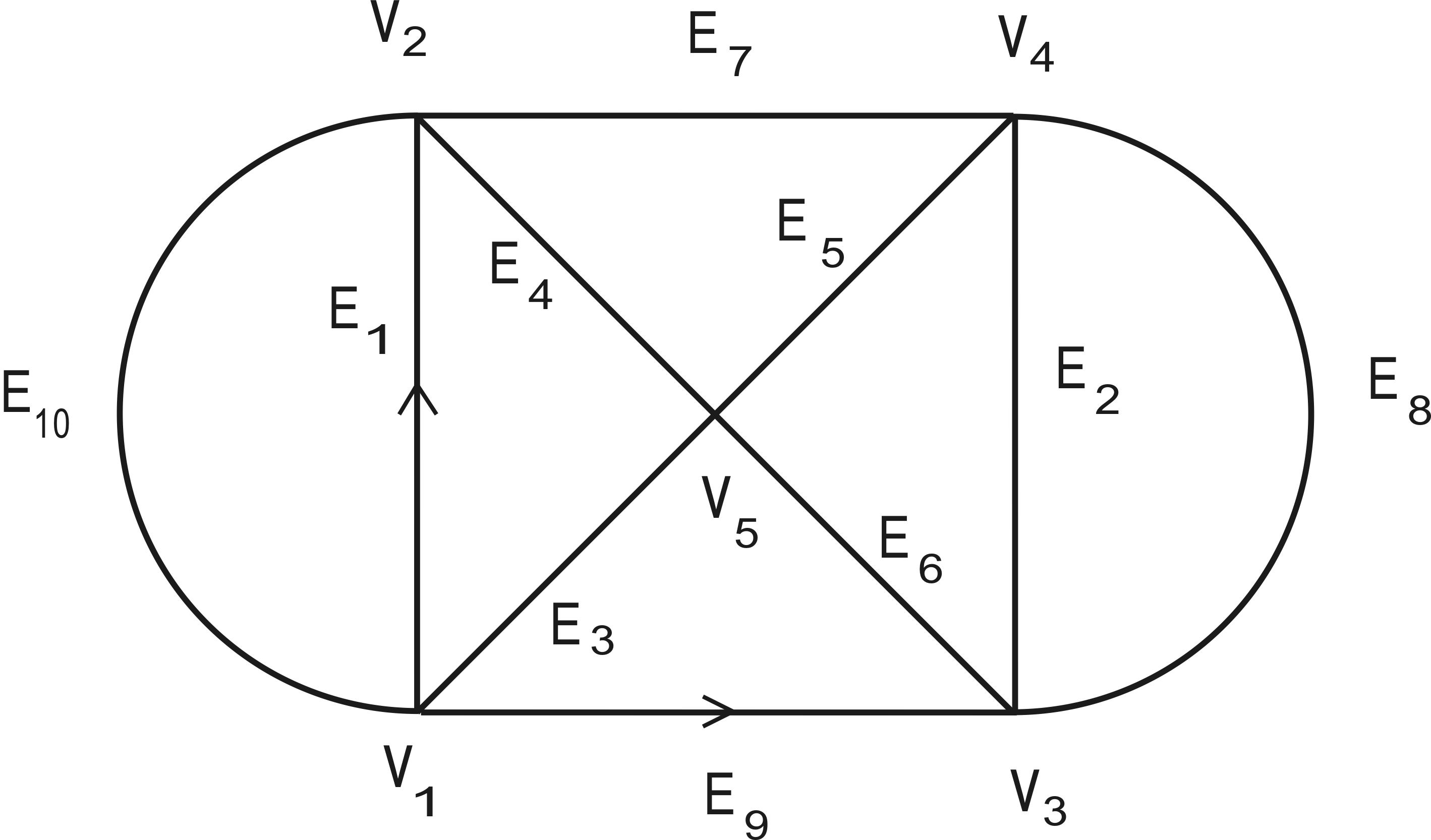}}
	\caption{The 1-skeleton of the canonical 2-spine $S$ of the 3-manifold at infinity of the complex hyperbolic triangle group  $\Delta_{3,4,4;\infty}$.}
\label{figure:1-skeleton4f}
\end{figure}

\begin{table}[htbp]
		\begin{tabular}{|c|c|}
			\hline
			Disk & Boundary path \\
			\hline
$b$ &
$E_{10}*E^{-1}_{1}*E_{10}*E_{7}*E_{5}*E^{-1}_{3}*E_{9}*E_{8}*E_{2}*E_{8}*$\\
 & $E^{-1}_{7}*E^{-1}_{4}*E_{6}*E^{-1}_{9}$
\\
			\hline
$BAB$ &  $E^{-1}_{8}*E^{-1}_{6}*E^{-1}_{5}$\\
			\hline
$bAB$ &  $E_{5}*E_{4}*E^{-1}_{1}*E_{9}*E^{-1}_{2}$ \\
			\hline
$Bab$ &  $E_{1}*E_{7}*E_{2}*E^{-1}_{6}*E^{-1}_{3}$\\
			\hline
$baB$ &  $E_{4}*E^{-1}_{10}*E_{3}$ \\
			\hline
		\end{tabular}
\caption{Boundary of the attaching disks in the 2-spine $S$  of the manifold at infinity of the complex hyperbolic triangle group  $\Delta_{3,4,4;\infty}$.}
\label{table:4disk}
\end{table}

\begin{table}[htbp]
		\begin{tabular}{|c|c|}
			\hline
			Disk& Relation   \\
			\hline
$b$ &  $\phi_3\phi^{-1}_6\phi^{-1}_4\phi^{-1}_6\phi_1\phi^{-1}_3\phi_5\phi_2\phi_5\phi^{-1}_1$
\\
			\hline
$BAB$ & $\phi_5$\\
			\hline
$bAB$ &  $\phi^{-1}_4\phi^{-1}_3\phi^{-1}_2$\\
			\hline
$Bab$ & $\phi_4\phi_1\phi_2$\\
			\hline
$baB$ & $\phi_6$\\
			\hline
		\end{tabular}
\caption{Relation of the attaching disks  in the 2-spine $S$  of the manifold at infinity of the  hyperbolic triangle group  $\Delta_{3,4,4;\infty}$.}
\label{table:4relation}
\end{table}

 In Table \ref{table:4relation}, we give the relations of the attaching disks  in the 2-spine $S$  of the manifold at infinity, which can be obtained from Table \ref{table:4disk}.
 From Table \ref{table:4relation} we get a presentation of the fundamental group of the 2-spine $S$, and so the fundamental group of $M$, the manifold at infinity of $\Delta_{3,4,4;\infty}$.  Then use Magma, we can simplify it, $$\pi_{1}(S)=\pi_{1}(M)=\langle x_1, x_2 | x_2^{-1} x_1^{-1}x_2^{2}x_1^{-1}x_2^{-1}x_1x_2x_1x_2x_1\rangle.$$

Let $m038$ be the one-cusped hyperbolic 3-manifold in the Snappy  Census \cite{CullerDunfield:2014}, which has volume 3.17729327860...
and $$\pi_{1}(m038)=\langle x, y |x^3y^{-1}x^{-1}y^3x^{-1}y^{-1}\rangle.$$

Using Magma, it is easy to see the groups above are isomorphic. This finishes the proof of Theorem \ref{thm:main} for  $\Delta_{3,4,4;\infty}$.

\section{Manifold at infinity of the complex hyperbolic triangle group  $\Delta_{3,4,6;\infty}$}

In this section, we study the manifold  at infinity of the even subgroup of the complex hyperbolic triangle group  $\Delta_{3,4,6;\infty}$. We follow the notation used in the Section 5.  Let $I_{1}$, $I_{2}$ and $I_{3}$ be the matrices in Section 3. We have $I_{1}I_{2}$ is parabolic, $(I_{2}I_{3})^3=(I_{3}I_{1})^4=id$, and $(I_{1}I_{3}I_{2}I_{3})^6=(I_{1}I_{2}I_{3}I_{2})^6=id$.
Let $A=I_{1}I_{2}$ and  $B=I_{2}I_{3}$, then $B^3=id$, $(AB)^4=id$ and  $(AB^2)^6=id$.  Let $\Gamma_{1}$ be the even subgroup of  $\Delta_{3,4,6;\infty}$ generated by $A$ and $B$,  and we write $a,b$ for $A^{-1}, B^{-1}$ respectively.

\subsection{The Ford domain of the complex hyperbolic triangle group  $\Delta_{3,4,6;\infty}$}

 We omit the detailed  information about finite ridges, even through they are crucial to use the Poincar\'e polyhedron theorem, but these are similar to Section \ref{section:comb344}.
 We study the ideal boundary  of the  Ford domain of the complex hyperbolic triangle group  $\Delta_{3,4,6;\infty}$ centered at the fixed point $q_{\infty}$ of the parabolic element $A$. We give the detailed information on spinal spheres and  infinite ridges, which is enough to get the 3-manifold  at infinity.

 Let $S^*$ be the set of  the  essential words $$\{B, b, BAB, bab, BAb, Bab, BaB, bAb, bAB, baB, BaBaB, bAbAb \}$$ and $D_{S^*}$ be  the intersection of the exteriors of the isometric spheres of the essential words in $S^*$ and all their translates by powers of $A$. The ideal boundary of the (partial) Ford domain  $D_{S^*}$ is  $E$, that is $E=\partial_{\infty}D_{S^*}$, and the boundary of the 3-manifold  $E$ is $C$, that  is $C=\partial E$. $C$ is an infinite annulus.

\begin{figure}[htbp]
	\scalebox{0.25}[0.25]{\includegraphics {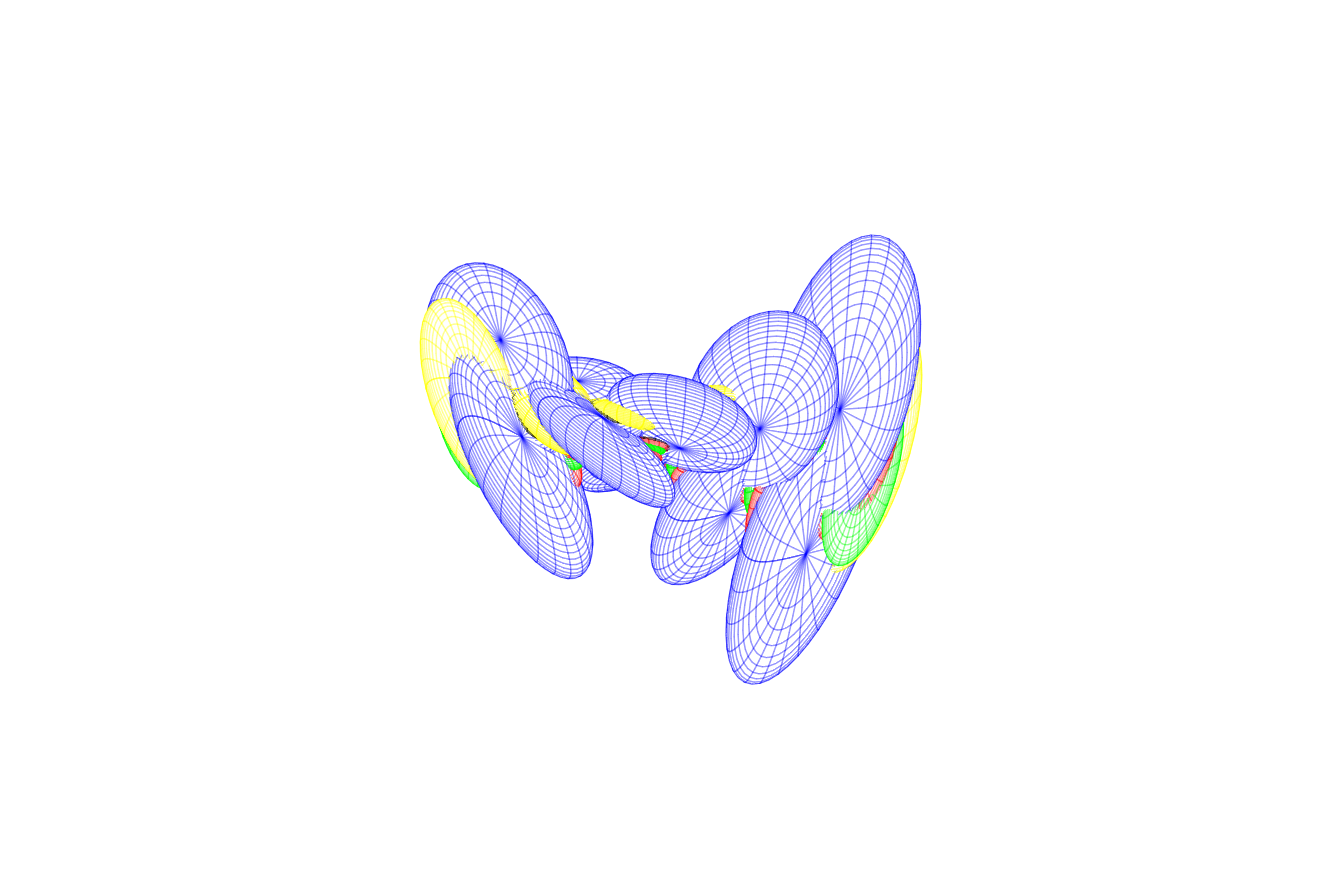}}
	\caption{Realistic view   of  the boundary of the ideal boundary of the Ford domain of the complex hyperbolic triangle group  $\Delta_{3,4,6;\infty}$.}
\label{figure:f6f}
\end{figure}

 Figure \ref{figure:f6f}  is  just for motivation, we will give rigorous   infinite ridge circles later.

We will show $C$ are contributed by the spinal spheres as follows:

  1. $\langle A \rangle$-translations of $\{B, b\}$, which are the  blue  colored polygons in Figure \ref{figure:f6f};

  2.  $\langle A \rangle$-translations of $\{BAB, bab\}$,    which are the  yellow  colored polygons in Figure \ref{figure:f6f};

  3.  $\langle A \rangle$-translations of $\{BAb, Bab\}$, which are the  pink colored   polygons in Figure \ref{figure:f6f};

  4.  $\langle A \rangle$-translations of $\{BaB, bAb\}$,  which are the  red colored   polygons in Figure \ref{figure:f6f};

  5.  $\langle A \rangle$-translations of $\{bAB, baB\}$,  which are the  black  colored   polygons in Figure \ref{figure:f6f};

  6.  $\langle A \rangle$-translations of $\{BaBaB, bAbAb\}$,  which are the  green colored   polygons in Figure \ref{figure:f6f}.

\begin{table}[htbp]
		\begin{tabular}{|c|c|}
			\hline
			Spinal sphere & Adjacent spinal spheres in  clockwise order  \\
			\hline

$baB$ & $[b]_{20}$,   $[BaB]_{6}$,   $[B]_{5}$\\
\hline
$bAB$ & $[B]_{5}$,   $[b]_{20}$,   $Ababa$\\
\hline
$BAb$ & $[b]_{5}$,   $[B]_{20}$,   $[bAb]_{6}$\\
\hline
$Bab$ & $[b]_{5}$,   $[B]_{20}$, $bab$\\
\hline
$(BAB)_{1}$ & $[B]_{20}$,   $[Aba]_{5}$, $ABaba$,  $[ABa]_{20}$, $[Aba]_{20}$\\
\hline
$(bab)_{1}$ & $[b]_{20}$,   $[aBA]_{20}$, $[b]_{5}$, $Bab$, $[B]_{20}$\\
\hline
$(BAB)_{2}$ & $[b]_{20}$,   $[B]_{20}$,  $[Aba]_{20}$, $[B]_{5}$, $bAB$\\
\hline
$(bab)_{2}$ & $[aBA]_{5}$,   $abABA$, $[abA]_{20}$, $[aBA]_{20}$, $[b]_{20}$\\
\hline
 $(BaB)_{1}$ & $[B]_{20}$,  $BaBaB$, $[abA]_{20}$,\\
\hline
$(bAb)_{1}$ & $[b]_{20}$,  $ABaBaBa$, $[ABa]_{20}$\\
\hline
 $(BaB)_{2}$ & $[abA]_{20}$,  $[B]_{5}$ $baB$, $[b]_{20}$, $[B]_{20}$,  $BaBaB$\\
\hline
$(bAb)_{2}$ & $[b]_{5}$,  $BAb$, $[B]_{20}$, $[b]_{20}$,  $ABaBaBa$, $[ABa]_{20}$\\
\hline
 $(BaBaB)_{1}$ & $[B]_{20}$,  $[abAbA]_{3}$, $[abA]_{20}$, $[BaB]_{6} (near~~baB)$\\
\hline
$(bAbAb)_{1}$ & $[ABa]_{20}$,  $[bAb]_{3}$, $[b]_{20}$, $[ABaBa]_{6} (near~~ AbaBA)$ \\
\hline
$(BaBaB)_{2}$ & $[abA]_{20}$,  $[BaB]_{3}$, $[B]_{20}$, $[abAbA]_{6} (near~~ aBAbA)$\\
\hline
 $(bAbAb)_{2}$ & $[b]_{20}$,  $[ABaBa]_{3}$, $[ABa]_{20}$, $[bAb]_{6} (near~~ BAb)$ \\
\hline	
$(B)_{1}$ & $[b]_{20}$, $[bAb]_{6}$, $BAb$, $[b]_{5}$, $[Aba]_{5}$, $BAB$,
$[Aba]_{20}$,  \\ & $BAB$, $[b]_{20}$, $bab$, $Bab$, $[b]_{5}$,
$[abA]_{5}$,  $[abAbA]_{6}$,  \\ & $BaBaB$, $[BaB]_{3}$, $[abA]_{20}$, $[abAbA]_{3}$, $[BaBaB]$, $[BaB]_{6}$\\
\hline	
 $(b)_{1}$  &  $[B]_{20}$, $[BaB]_{6}$, $baB$, $[B]_{5}$, $[aBA]_{5}$, $bab$, $[aBA]_{20}$, \\ &
 $bab$, $[B]_{20}$, $BAB$,  $bAB$,  $[B]_{5}$, $[ABa]_{5}$, $[ABaBa]_{6}$, \\ &  $bAbAb$, $[bAb]_{3}$, $[ABa]_{20}$, $[ABaBa]_{3}$,  $bAbAb$,  $[bAb]_{6}$  \\
\hline	
$(B)_{2}$ & $[abA]_{20}$, $[aBA]_{5}$, $[b]_{20}$, $baB$, $[BaB]_{6}$ \\
\hline
$(b)_{2}$  &  $[B]_{20}$,   $Bab$, $bab$,  $[aBA]_{20}$,  $[abA]_{5}$\\
\hline
$(B)_{3}$ & $[b]_{20}$, $bAB$, $BAB$, $[Aba]_{20}$, $[ABa]_{5}$\\
\hline
$(b)_{3}$ & $[B]_{20}$, $BAb$, $[bAb]_{6}$, $[ABa]_{20}$, $[Aba]_{5}$ \\
\hline		
\end{tabular}
\caption{Adjacent relations in the boundary of the ideal boundary of the Ford domain of   $\Delta_{3,4,6;\infty}$.}
\label{table:346adjacet}
\end{table}

In Table \ref{table:346adjacet}, we give the adjacent relations of the spinal spheres   in the ideal boundary of the Ford domain boundary of   $\Delta_{3,4,6;\infty}$.

But this is more complicated than the  $\Delta_{3,4,4;\infty}$ case. For example the spinal sphere of  $B$ give 3 parts in this boundary,  one icosagon and two pentagons, we denote them as $(B)_{1}$,  $(B)_{2}$ and  $(B)_{3}$ in Table \ref{table:346adjacet}. Where it is easy to see $(B)_{1}$ is an icosagon,   $(B)_{2}$ and  $(B)_{3}$ are pentagons. Moreover, we also use $[B]_{20}$ to denote $(B)_{1}$, and   $[B]_{5}$ to  denote any of $(B)_{2}$ or  $(B)_{3}$, and we can know which it is since the adjacent spinal spheres of  $(B)_{2}$ and  $(B)_{3}$ are different. By this we mean $(B)_{3}$ is adjacent to the spinal spheres of  $bAB$ and $BAB$,   $(B)_{2}$ is adjacent to  the spinal sphere of $baB$. Each of the spinal spheres of  $\{bAB, baB\}$ give a triangle. Similarly,  $[bAb]_{6}$  means part of the spinal sphere of  $bAb$ which is a hexagon.

For concrete information about  adjacent relations of the spinal  spheres,  see also Figure \ref{figure:comb6f}. In Figure \ref{figure:comb6f}, we only give some of the labels, which means the region is a part of the spinal sphere with that label. The remaining labels can be obtained from the partial labels of  in Figure \ref{figure:comb6f} and Table \ref{table:346adjacet}.
 Figure \ref{figure:comb6f} gives a tilling of the plane, we glue together the tilling when they have the same label,  we get the infinite annulus $C=\partial E$,  the $A$-action is just  a  negative horizontal translation.

\begin{figure}[htbp]
	\scalebox{0.25}[0.25]{\includegraphics {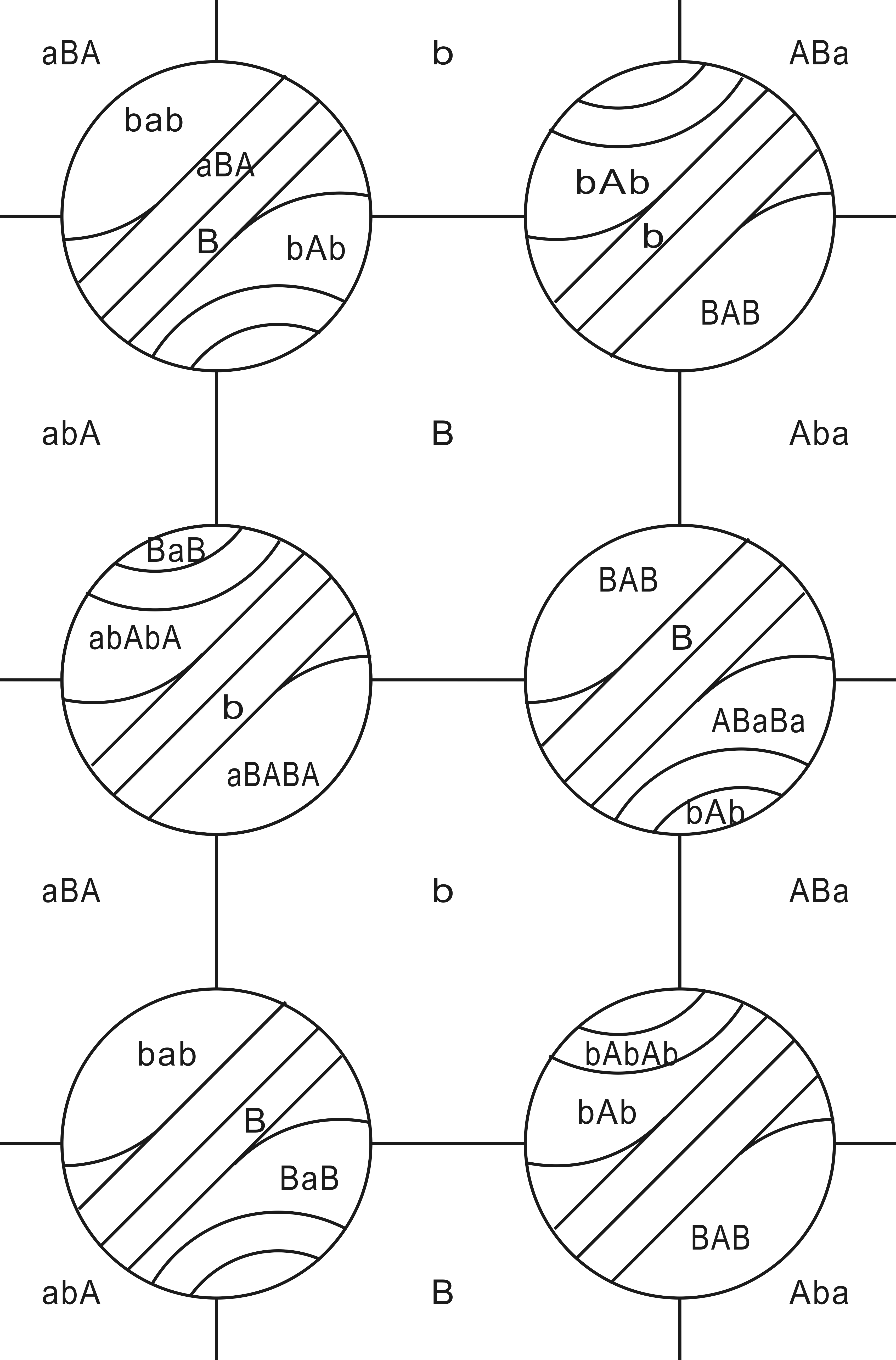}}
	\caption{A combinatorial picture of the  boundary of the ideal boundary of the Ford domain of the complex hyperbolic triangle group  $\Delta_{3,4,6;\infty}$.}
\label{figure:comb6f}
\end{figure}

\begin{table}[htbp]
		\begin{tabular}{|c|c|}
			\hline
			Number & Ridge relation  \\
			\hline
$R_{1}$ & $(r_1, b\cap B)\rightarrow (r_{62}, b\cap B)\rightarrow$  \\
& $(r_{9}, Aba\cap ABa) \rightarrow (r_{32}, b \cap B) \rightarrow (r_1, b\cap B)$
\\
			\hline
$R_{2}$ & $ (r_2, bAb\cap b)\rightarrow (r_{33}, B\cap abA)\rightarrow$   \\
&    $(r_{8}, ABa\cap b)\rightarrow (r_{65}, ABa\cap ABaBa) \rightarrow (r_2, bAb\cap b)$
\\
			\hline
$R_{3}$ & $(r_3, bAbAb\cap b)\rightarrow (r_{34}, B\cap abAbA)\rightarrow$  \\
& $(r_{7}, ABa\cap bAb) \rightarrow (r_{69}, ABaBa \cap bAbAb) \rightarrow (r_3, bAbAb\cap b)$ \\
			\hline
$R_{4}$ & $(r_4, ABaBa\cap b)\rightarrow (r_{35}, B\cap BaBaB)\rightarrow (r_{6}, ABa\cap bAbAb)$  \\
& $(r_{71}, bAbAb\cap bAb) \rightarrow (r_4, ABaBa\cap b)$ \\
			\hline
$R_{5}$ & $(r_5, ABaBa\cap ABa)\rightarrow (r_{36}, BaB\cap B)\rightarrow (r_{22}, b\cap ABa) $ \\
& $ \rightarrow (r_{37}, abA \cap B)  \rightarrow (r_{72}, b \cap bAb) \rightarrow (r_5, ABaBa\cap ABa)$ \\
			\hline
$R_{6}$ & $(r_{10}, ABaba\cap ABa)\rightarrow (r_{31}, Bab\cap B)\rightarrow (r_{23}, b\cap BaB)$ \\
& $ \rightarrow (r_{18}, Aba\cap ABaBa) \rightarrow (r_{44}, bAb \cap BAb) \rightarrow (r_{10}, ABaba\cap ABa)$ \\
			\hline
$R_{7}$ & $(r_{11}, BAB\cap ABa)\rightarrow (r_{30}, aBABa\cap B)\rightarrow (r_{24}, b\cap baB) \rightarrow$ \\
& $(r_{17}, Aba \cap AbaBa) \rightarrow (r_{58}, bAB\cap BAB) \rightarrow (r_{11}, BAB\cap ABa)$
\\
			\hline
$R_{8}$ & $(r_{12}, BAB\cap Aba)\rightarrow (r_{29}, aBABA\cap b)\rightarrow (r_{50}, B\cap Aba) \rightarrow$ \\
& $ (r_{59}, B\cap BAB) \rightarrow (r_{12}, BAB\cap Aba)$
\\
			\hline
$R_{9}$ & $(r_{13}, B\cap Aba)\rightarrow (r_{28}, aBA\cap b)\rightarrow (r_{54}, B\cap BAB) \rightarrow$ \\
& $(r_{14}, BAB\cap AbA) \rightarrow (r_{27}, bab\cap b) \rightarrow (r_{13}, B\cap Aba)$
\\
			\hline

$R_{10}$ & $(r_{15}, B\cap Aba)\rightarrow (r_{26}, aBA\cap b)\rightarrow (r_{55}, B\cap BAB) \rightarrow$  \\
& $ (r_{52}, BAB\cap Aba) \rightarrow (r_{15}, B\cap Aba)$
\\
	\hline
$R_{11}$ & $(r_{16}, ABa\cap Aba)\rightarrow (r_{25}, B\cap b)\rightarrow$ \\
& $(r_{56}, B\cap b) \rightarrow (r_{48}, B \cap b) \rightarrow (r_{16}, ABa\cap Aba)$
\\
			\hline
$R_{12}$ & $(r_{19}, ABaBa\cap ABa)\rightarrow (r_{40}, BaB\cap B)\rightarrow$ \\
& $(r_{64}, b\cap ABa) \rightarrow (r_{45}, b \cap bAb) \rightarrow (r_{19}, ABaBa\cap ABa)$
\\
			\hline
$R_{13}$ & $(r_{20}, bAbAb\cap ABa)\rightarrow (r_{39}, BaBaB\cap B)\rightarrow$ \\
& $(r_{68}, b\cap ABaBa) \rightarrow (r_{42}, bAb \cap bAbAb) \rightarrow (r_{20}, bAbAb\cap ABa)$
\\
			\hline
$R_{14}$ & $(r_{21}, bAb\cap ABa)\rightarrow (r_{38}, abAbA\cap B)\rightarrow$  \\
& $(r_{70}, b\cap bAbAb) \rightarrow (r_{41}, bAbAb \cap ABaBa) \rightarrow (r_{21}, bAb\cap ABa)$
\\
			\hline
$R_{15}$ & $(r_{43}, bAb\cap B)\rightarrow (r_{60}, b\cap bAB)\rightarrow (r_{67}, AbaBa\cap ABaBa)$ \\
&  $\rightarrow (r_{43}, bAb\cap B)$ \\

	\hline	
$R_{16}$ & $(r_{46}, BAb\cap B)\rightarrow (r_{57}, b\cap BAB)\rightarrow (r_{53}, BAB\cap ABaba)$ \\
&  $ \rightarrow (r_{46}, BAb\cap B)$ \\
	\hline
$R_{17}$ & $(r_{47}, BAb\cap b)\rightarrow (r_{63}, ABa\cap B)\rightarrow (r_{51}, Aba\cap ABaba)$ \\
&  $ \rightarrow (r_{47}, BAb\cap b)$ \\
		\hline
$R_{18}$ & $(r_{49}, b\cap Aba)\rightarrow (r_{61}, B\cap bAB)\rightarrow (r_{66}, AbaBa\cap ABa)$ \\
&  $ \rightarrow (r_{49}, b\cap Aba)$ \\
			\hline
		\end{tabular}
\caption{Infinite ridge relations  of the Ford domain of $\Delta_{3,4,6;\infty}$, part I.}
\label{table:ridge346}
\end{table}

From  Figure \ref{figure:comb6f}, we have Tables \ref{table:ridge346} and \ref{table:circle346}, which give all the infinite ridge circle relations  of the Ford domain of $\Delta_{3,4,6;\infty}$.
We should note that the side-pairings and the ridge relations are more complicated than the $\Delta_{3,4,4;\infty}$ case.
For example, the circle relation for the ridge $R_{1}$, which is $B*a*A*B*B$: Firstly, from $(r_1, b\cap B)\rightarrow (r_{62}, b\cap B)$ we get the word  $B$, which  maps the isometric sphere of $B$ to the isometric sphere of $b$. Secondly,  from  $(r_{62}, b\cap B)$, we  consider the map $B$ which maps   the isometric sphere of $B$ to the isometric sphere of $b$, even through we get $(r_{32}, b \cap B)$ now, but note that here we don't get the side-paring from our chosen fundamental domain  in Subsection \ref{subsection:346spine}, we must perform the translation give by $A$, then we get $(r_{9}, Aba\cap ABa)$, now we get the side-pairing from $(r_{32}, b \cap B)$ to  $(r_{9}, Aba\cap ABa)$. In other words, $A*B$ is a side-pairing from  $(r_{62}, b\cap B)$  to $(r_{9}, Aba\cap ABa)$. Together with the first map $B$ from  $(r_1, b\cap B)$ to $(r_{62}, b\cap B)$, now we have the word  $A*B*B$; Thirdly,  using the $A^{-1}=a$-translation,  from $(r_{9}, Aba\cap ABa) $ we get $ (r_{32}, b \cap B)$, so we have the word $a*A*B*B$ now. Fourthly,  from  $(r_{32}, b \cap B) \rightarrow (r_{1}, b\cap B)$, we have the word $B$. In total, we have the word  $B*a*A*B*B$, which is $B^3$, so the circle relation is $B^3=id$. We omit details for other ridges.

\begin{table}[htbp]
		\begin{tabular}{|c|c|c|}
			\hline
			Number& Ridge cycle  & Cycle relation \\
			\hline
$R_{1}$ &   $B*a*A *B*B$ &    $B^3=id$\\
			\hline
$R_{2}$   & $a*(ABaBa)*A*b*A*b$ &  $id$\\
			\hline
$R_{3}$  & $A*(bAbAb)*A*(bAb)*A*b$ &  $(Ab)^{6}=id$\\
			\hline
$R_{4}$   & $A*(bAb)*A*(bAbAb)*A*b$ &  $(Ab)^{6}=id$
\\
			\hline
$R_{5}$   & $A*(bAb)*B*a*B*a$ &  $id$
\\
			\hline
$R_{6}$   & $A*(BAb)*a*(ABaBa)*A*B*a$ &  $B^3=id$
\\
	\hline
$R_{7}$  & $A*(BAB)*a*(AbaBa)*A*B*a$ &  $B^3=id$
\\
			\hline
$R_{8}$   & $A*(BAB)*a*(Aba)*b*a$ &  $id$
\\
			\hline
$R_{9}$   & $b*a*A*(BAB)*b*a$ &  $id$
\\
			\hline

$R_{10}$   & $a*(Aba)*A*BAB*b*a$ &  $id$
\\
	\hline
$R_{11}$    & $A*b*b*b*a$ &  $b^{3}=id$
\\
			\hline
$R_{12}$  & $A*(bAb)*a*(ABa)*B*a$ &  $id$
\\
			\hline
$R_{13}$  & $A*(bAbAb)*a*(ABaBa)*B*a$ &  $id$
\\
			\hline
$R_{14}$ & $a*(ABaBa)*A*(bAbAb)*B*a$ &  $id$
\\
			\hline
$R_{15}$  & $a*(ABaBa)*A*(bAB)*B$ &  $B^{3}=id$\\

	\hline	
$R_{16}$  & $a*(ABaba)*A*(BAB)*B$ &  $B^{3}=id$\\
	\hline
$R_{17}$ & $a*(ABaba)*A*B*A*b$ &  $id$\\
		\hline
$R_{18}$  & $a*(ABa)*A*(bAB)*a*(Aba)$ &  $id$\\
			\hline
		\end{tabular}
\caption{Infinite ridge relations  of the Ford domain of $\Delta_{3,4,6;\infty}$, part II.}
\label{table:circle346}
\end{table}

Moreover, by studying the intersections of the spinal spheres,  we have
\begin{prop} \label{prop: 346para} There is no parabolic fixed point  in the boundary of the Ford domain $\Delta_{3,4,6;\infty}$.
\end{prop}

From Tables  \ref{table:ridge346} and \ref {table:circle346}, and the information about bounded ridges that we omitted as in Section \ref{section:comb344}, we can use the Poincar\'e polyhedron theorem to see that  the partial Ford domain bounded by $\langle A \rangle$-translations of the isometric  spheres of $\{B, b\}$, $\{BAB, bab\}$, $\{BAb, Bab\}$, $\{BaB, bAb\}$,  $\{bAB, baB\}$  and  $\{BaBaB, bAbAb\}$ is in fact the Ford domain. So we can study the manifold at infinity of  $\Delta_{3,4,6;\infty}$.

\subsection{A 2-spine $S$ of the 3-manifold at infinity of $\Delta_{3,4,6;\infty}$} \label{subsection:346spine}

We now consider the canonical 2-spine $S$ of the 3-manifold at infinity of $\Delta_{3,4,6;\infty}$.

\begin{figure}[htbp]
	\scalebox{0.35}[0.35]{\includegraphics {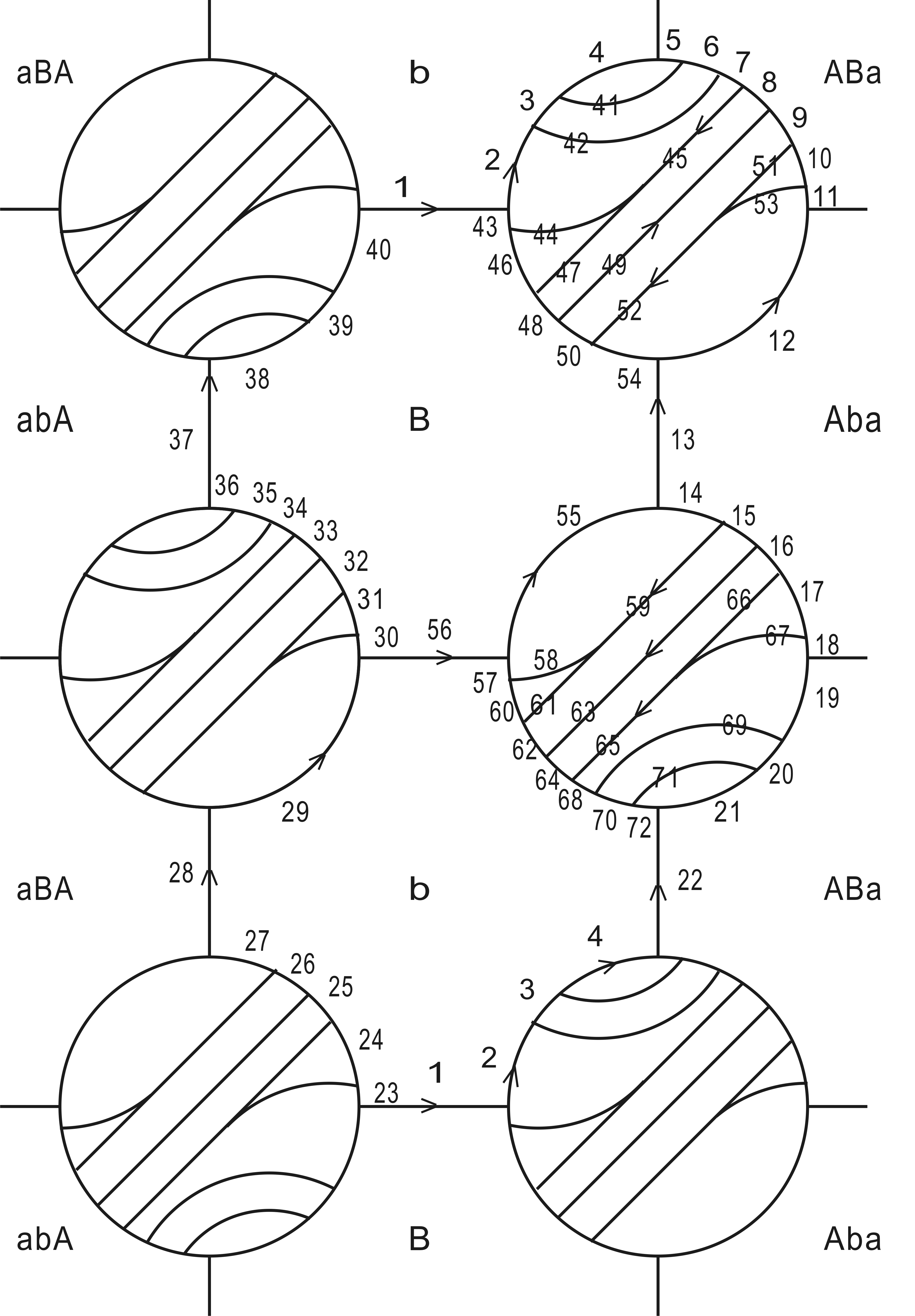}}
	\caption{Fundamental domain and edge cycles in the boundary of  the ideal boundary of the Ford domain of the complex hyperbolic triangle group  $\Delta_{3,4,6;\infty}$.}
	\label{figure:boundary6forientation}
\end{figure}

As in the case of  $\Delta_{3,4,4;\infty}$, each infinite ridge in the boundary of the Ford domain gives us an edge in the canonical  2-spine of the 3-manifold at infinity. We take a fundamental domain of $A$ acting on the infinite annulus $C$, which in Figure \ref{figure:boundary6forientation} is the largest region  bounded by the edges with labels. It consists of:

(1). Two big icosagons correspond to the spinal spheres of $\{B, b\}$.

(2). Two hexagons correspond to the spinal spheres of  $\{ABaBa, bAb\}$.

(3). Four  pentagons  correspond to the spinal spheres of the elements  $\{B,b\}$ and $\{Aba, ABa\}$. Two more  pentagons  correspond to the spinal spheres of  $BAB$. Note that since $(AB)^4=id$,  $BAB=ababa$ and $Ababa$ have the same isometric sphere.

(4). Two quadrilateral  correspond to the spinal spheres of  $bAbAb$. Note that since $(AB^2)^6=id$,  $bAbAb=ABaBaBa$, the isometric sphere of $BaBaB$ is just an $A$-translation of the isometric sphere of $bAbAb$.

(5). Six triangles  correspond to the spinal spheres of   $\{bAb, ABaBa\}$, $\{BAb,ABaba\}$ and  $\{bAB,AbaBa\}$.

\begin{table}[htbp]\label{table: 6edge}
		\begin{tabular}{|c|c|}
			\hline
			Edge & Equivalence class   \\
			\hline
$E_1$ & $e_{1}$, $e_{62}$, $e_{9}$, $e_{32}$\\
			\hline
$E_2$ & $e_{2}$, $e_{33}$, $e_{8}$, $e_{65}$\\
			\hline
$E_3$ & $e_{3}$, $e_{34}$, $e_{7}$, $e_{69}$\\
			\hline
$E_4$ & $e_{4}$, $e_{35}$, $e_{6}$, $e_{71}$\\
			\hline
$E_5$ & $e_{5}$, $e_{36}$, $e_{22}$, $e_{37}$, $e_{72}$\\
			\hline
$E_6$ & $e_{10}$, $e_{31}$, $e_{23}$, $e_{18}$, $e_{44}$\\
			\hline
$E_7$ & $e_{11}$, $e_{30}$, $e_{24}$, $e_{17}$, $e_{58}$\\
\hline
$E_8$ & $e_{12}$, $e_{29}$, $e_{50}$, $e_{59}$\\
			\hline
$E_9$ & $e_{13}$, $e_{28}$, $e_{54}$, $e_{14}$, $e_{27}$\\
			\hline
$E_{10}$ & $e_{15}$, $e_{26}$, $e_{55}$, $e_{52}$\\
			\hline
$E_{11}$ & $e_{16}$, $e_{25}$, $e_{56}$, $e_{48}$\\
			\hline
$E_{12}$ & $e_{19}$, $e_{40}$, $e_{64}$, $e_{45}$\\
			\hline
$E_{13}$ & $e_{20}$, $e_{39}$, $e_{68}$, $e_{42}$\\
			\hline
$E_{14}$ & $e_{21}$, $e_{38}$, $e_{70}$, $e_{41}$\\
			\hline
$E_{15}$ & $e_{43}$, $e_{60}$, $e_{67}$\\
			\hline
$E_{16}$ & $e_{46}$, $e_{57}$, $e_{53}$\\
			\hline
$E_{17}$ & $e_{47}$, $e_{63}$, $e_{51}$\\
			\hline
$E_{18}$ & $e_{49}$, $e_{61}$, $e_{66}$\\
			\hline
		\end{tabular}
\caption{Equivalence classes of edges in 1-skeleton of the canonical 2-spine $S$ of the 3-manifold at infinity of $\Delta_{3,4,6;\infty}$.}
\label{table:edge6}
\end{table}

Note that in Figure \ref{figure:boundary6forientation},  we labeled the edges from $1$ to $72$, so there are $72$ edges in total.

 Now Table \ref{table:edge6} is obtained from Figure \ref{figure:boundary6forientation},  Tables \ref{table:ridge346} and \ref{table:circle346}, it gives the equivalence classes of edges in the canonical 2-spine $S$ of the 3-manifold at infinity  of $\Delta_{3,4,6;\infty}$, a pair of  2-cells in the fundamental domain with inverse labels  as in Figure \ref{figure:boundary6forientation} give a 2-cell of the 2-spine $S$.

In Table \ref{table:disk6}, we give the information of attaching disks of the canonical 2-spine $S$ of the 3-manifold at infinity of $\Delta_{3,4,6;\infty}$, which is obtained from Figure \ref{figure:boundary6forientation} and Table \ref{table:edge6}. The orientations of the edges and the equivalence classes of end points of the edges is showed in Table \ref{table:vertex6}.

\begin{table}[htbp]\label{table: 6boundary}
	\begin{tabular}{|c|c|}
		\hline
		Disk & Boundary path \\
		\hline
		
		$[B]_{20}$  &
		$E_{1}*E^{-1}_{15}*E^{-1}_{16}*E^{-1}_{11}*E^{-1}_{8}*(E^{-1}_{9})^{2}*E^{-1}_{10}*E^{-1}_{11}*$
		\\ &   $E_{7}*E_{6}*E_{1}*E_{2}*E_{3}*E_{4}*(E_{5})^{2}*E_{14}*E_{13}*E_{12}$
		\\
		\hline
		$ABaBa$   &   $E_{4}*E^{-1}_{5}*E_{14}$\\
		\hline
		$bAbAb$ &    $E_{3}*E^{-1}_{14}*E^{-1}_{4}*E_{13}$ \\
		\hline
		$bAb$ &  $E_{2}*E^{-1}_{13}*E^{-1}_{3}*E_{12}*E^{-1}_{6}*E_{15}$\\
		\hline
		$BAb$ &   $E_{6}*E^{-1}_{17}*E_{16}$\\
		\hline
		$[b]_{5}$ &  $E^{-1}_{12}*E^{-1}_{2}*E^{-1}_{18}*E_{11}*E_{17}$\\
		\hline
		$Aba$ &     $E_{1}*E^{-1}_{18}*E^{-1}_{8}*E^{-1}_{10}*E_{17}$ \\
		\hline
		$Ababa$ &     $E_{17}*E^{-1}_{6}*E^{-1}_{16}$ \\
		\hline
		$BAB$ &    $E_{16}*E^{-1}_{7}*E^{-1}_{8}*E_{9}*E^{-1}_{10}$\\
		\hline
		$bAB$ &   $E_{18}*E^{-1}_{15}*E^{-1}_{7}$ \\
		\hline
		
	\end{tabular}
	\caption{Boundary of the attaching disks of the canonical 2-spine $S$ of the manifold at infinity of  $\Delta_{3,4,6;\infty}$.}
	\label{table:disk6}
\end{table}

\begin{table}[htbp]\label{table: 6vertices}
		\begin{tabular}{|c|c|}
			\hline
			Vertex & Equivalence class   \\
			\hline
$V_1$ & $E_{1}^{+}$, $E_{2}^{-}$, $E_{15}^{+}$, $E_{18}^{+}$\\
			\hline
$V_2$ & $E_{1}^{-}$, $E_{12}^{+}$, $E_{6}^{+}$, $E_{17}^{+}$\\
			\hline
$V_3$ & $E_{2}^{+}$, $E_{3}^{-}$, $E_{13}^{+}$, $E_{12}^{-}$\\
			\hline
$V_4$ & $E_{4}^{-}$, $E_{14}^{+}$, $E_{3}^{+}$, $E_{13}^{-}$\\
			\hline
$V_5$ & $E_{4}^{+}$, $E_{5}^{+}$, $E_{5}^{-}$, $E_{14}^{-}$\\
			\hline
$V_6$ & $E_{6}^{-}$, $E_{7}^{+}$, $E_{16}^{+}$, $E_{15}^{-}$\\
			\hline
$V_7$ & $E_{8}^{-}$, $E_{9}^{+}$, $E_{9}^{-}$, $E_{10}^{+}$\\
			\hline
$V_8$ & $E_{7}^{-}$, $E_{8}^{+}$, $E_{11}^{-}$, $E_{18}^{-}$\\
			\hline
$V_9$ & $E_{10}^{-}$, $E_{11}^{+}$, $E_{16}^{-}$, $E_{17}^{-}$\\
			\hline
			
		\end{tabular}
\caption{Equivalence classes of vertices  of the canonical  2-spine $S$ of the 3-manifold at infinity  of $\Delta_{3,4,6;\infty}$.}
\label{table:vertex6}
\end{table}

In Table \ref{table:vertex6}, for example, $E_{1}^{+}$ means the positive end of the oriented edge $E_{1}$, and  $E_{1}^{-}$ means the negative end of the oriented edge $E_{1}$.

From  Table \ref{table:vertex6}, we get Figure \ref{figure:1-skeleton6f}, which is the 1-skeleton of the canonical 2-spine $S$ of our 3-manifold at infinity of $\Delta_{3,4,6;\infty}$. From Figure \ref{figure:1-skeleton6f}, we consider the fundamental group of the graph.  In Figure \ref{figure:1-skeleton6f}, the orientation of  the edges are obtained from Table \ref{table:vertex6}:
edge $E_1$ is  from $V_2$ to $V_1$; edge $E_2$ is  from $V_1$ to $V_3$; edge $E_3$ is  from $V_3$ to $V_4$; edge $E_4$ is  from $V_4$ to $V_5$; edge $E_5$ is  from $V_5$ to $V_5$; edge  $E_6$ is  from $V_2$ to $V_6$; edge $E_7$ is  from $V_8$ to $V_6$;
edge $E_8$ is  from $V_7$ to $V_8$; edge $E_9$ is  from $V_7$ to $V_7$; edge $E_{10}$ is  from $V_9$ to $V_7$; edge $E_{11}$ is  from $V_8$ to $V_9$; edge  $E_{12}$ is  from $V_3$ to $V_2$; edge $E_{13}$ is  from $V_4$ to $V_3$; edge $E_{14}$ is  from $V_5$ to $V_4$; edge $E_{15}$ is  from $V_6$ to $V_1$; edge $E_{16}$ is  from $V_9$ to $V_6$; edge $E_{17}$ is  from $V_9$ to $V_2$; edge $E_{18}$ is  from $V_8$ to $V_1$.

\begin{figure}[htbp]
	\scalebox{0.25}[0.25]{\includegraphics {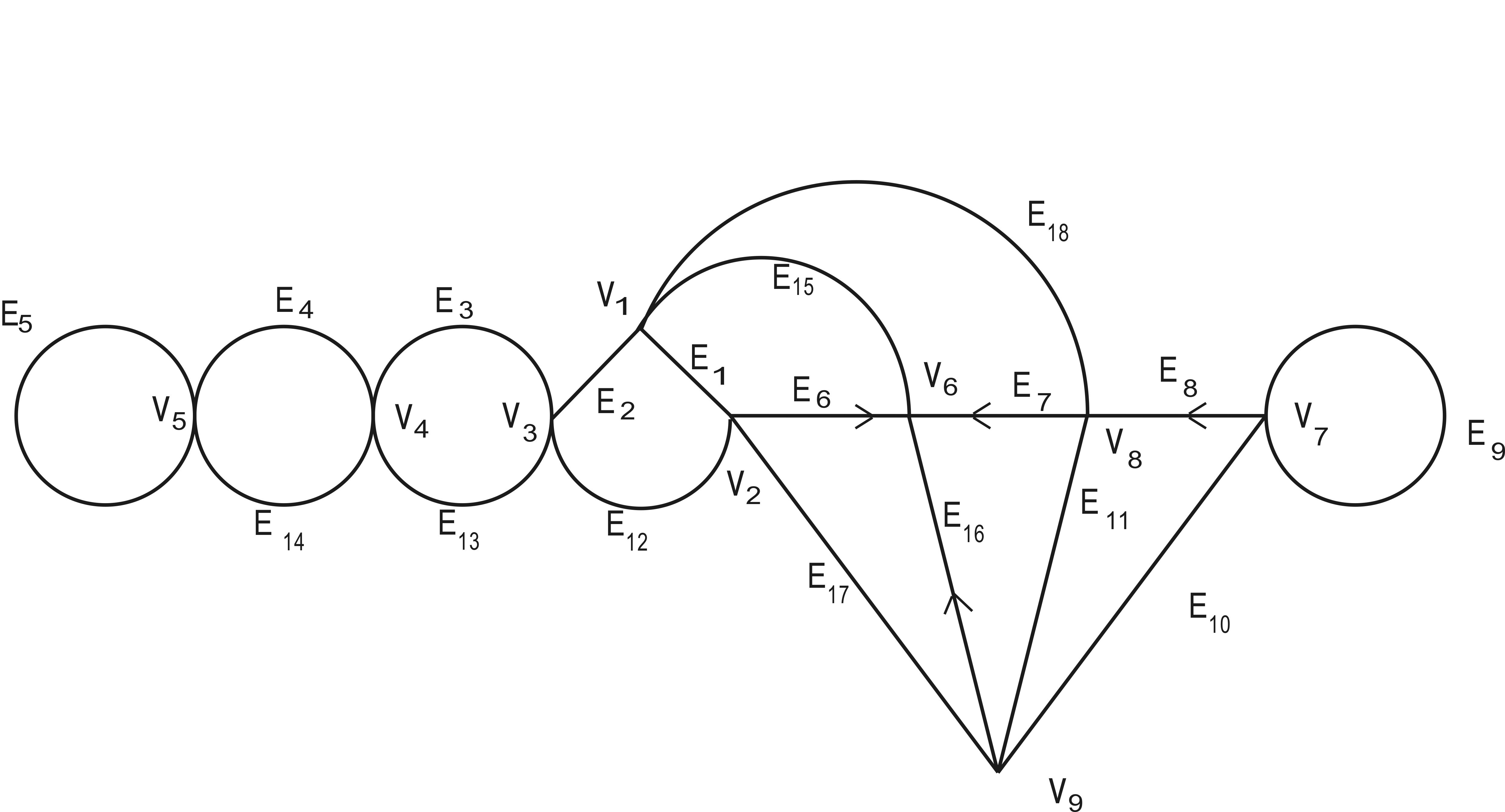}}
	\caption{The 1-skeleton of the canonical 2-spine $S$ of  the 3-manifold at infinity of  the complex hyperbolic triangle group $\Delta_{3,4,6;\infty}$.}
\label{figure:1-skeleton6f}
\end{figure}

We let $\phi_1=E_1*E_2*E_{12}$,   $\phi_2=E^{-1}_{12}*E_3*E_{13}*E_{12}$,  $\phi_3=E^{-1}_{12}*E_3*E_{4}*E_{14}*E^{-1}_{3}*E_{12}$,
 $\phi_4=E^{-1}_{12}*E_3*E_{4}*E_{5}*E^{-1}_{4}*E^{-1}_{3}*E_{12}$,
 $\phi_5=E_1*E^{-1}_{15}*E_{6}$,
$\phi_6=E_1*E^{-1}_{18}*E_{7}*E_{6}$,
  $\phi_7=E^{-1}_{17}*E_{16}*E_{6}$,
$\phi_8=E^{-1}_6*E^{-1}_{7}*E_{11}*E_{17}$, $\phi_9=E^{-1}_6*E^{-1}_{7}*E^{-1}_{8}*E^{-1}_{10}*E_{17}$,
$\phi_{10}=E^{-1}_6*E^{-1}_{7}*E^{-1}_{8}*E_{9}*E_{8}*E_{7}*E_{6}$
be closed loops in the graph.
Here $E^{-1}_i$ is the inverse path of $E_i$.
And it is easy to see $\phi_{i}$, $1 \leq i \leq 10$, is a  generator set of the fundamental group  of the graph with base the vertex  $V_2$.

\begin{table}[htbp]\label{table: 6relation}
		\begin{tabular}{|c|c|}
			\hline
			 Disk & Relation \\
			\hline
$[B]_{20}$ &
$\phi_{5}\phi^{-1}_7\phi^{-1}_{8}\phi^{-1}_{10}\phi^{-1}_{10}\phi_{9}\phi^{-1}_{8}\phi_{1}\phi_{4} \phi_{4} \phi_3\phi_{2}$
\\
			\hline
$ABaBa$ &  $\phi^{-1}_{4}\phi_{3}$\\
			\hline
$bAbAb$ &    $\phi^{-1}_{3}\phi_{2}$\\
			\hline
$bAb$ &  $\phi_1\phi^{-1}_{2}\phi^{-1}_{5}$ \\
			\hline
$BAb$ &  $\phi_7$ \\
			\hline
$[b]_{5}$ &  $\phi^{-1}_{1}\phi_6\phi_{8}$ \\
			\hline
$Aba$ &  $\phi_6\phi_{9}$  \\
			\hline
$Ababa$ &    $\phi^{-1}_{7}$  \\
			\hline
$BAB$ &  $\phi_7\phi_{10}\phi_9$\\
			\hline
$bAB$ &   $\phi_5\phi^{-1}_{6}$\\
			\hline

		\end{tabular}
\caption{Relations of  the attaching disks of the 2-spine $S$ of the manifold at infinity of $\Delta_{3,4,6;\infty}$.}
\label{table:6relation}
\end{table}

From Table \ref{table:disk6}, we get Table \ref{table:6relation}, then we get a presentation of the fundamental group  $\pi_{1}(S)$ of the 2-spine $S$, and so the fundamental group of $M$, the manifold at infinity of $\Delta_{3,4,6;\infty}$.

Let $s090$ be the one-cusped hyperbolic manifold in Snappy  Census \cite{CullerDunfield:2014}, it is hyperbolic with volume 3.89049957640...
and $$\pi_{1}(s090)=\langle x, y |x^5y^{-1}x^{-1}y^3x^{-1}y^{-1}\rangle.$$

Using Magma, it is easy to see $\pi_{1}(s090)$ and $\pi_{1}(S)$ above are isomorphic. This finishes the proof of Theorem \ref{thm:main} for  $\Delta_{3,4,6;\infty}$.

\section*{Acknowledgement}
This work was carried out while the second author was visiting the Fudan university under the financial support by the GaoFeng plan from School of   Mathematical Sciences, Fudan University. We thank Fudan University for excellent working conditions.

\bibliographystyle{amsplain}

\end{document}